\documentclass[11pt,a4paper]{amsart}
\usepackage{amsfonts,amsmath,amssymb}
\usepackage{hyperref}
\usepackage{latexsym,fullpage,amsfonts,amssymb,amsmath,amscd,graphics}
\usepackage[all]{xy}
\usepackage{amssymb,amsthm,amsxtra}
\usepackage[usenames]{color}
\usepackage{amscd}
\usepackage{amsthm}
\usepackage{amsfonts}
\usepackage{amssymb}
\usepackage{mathrsfs}
\usepackage{url}
\usepackage{bbm}
\usepackage{wasysym}
\usepackage{enumitem}
\usepackage[vcentermath]{youngtab}%
\theoremstyle{plain}
\newtheorem*{theorem*}{Theorem}
\newtheorem*{remark*}{Remark}
\newtheorem*{example*}{Example}
\newtheorem{lemma}{Lemma}[subsection]
\newtheorem{proposition}[lemma]{Proposition}
\newtheorem{corollary}[lemma]{Corollary}
\newtheorem{theorem}[lemma]{Theorem}

\newtheorem*{conjecture*}{Conjecture}

\newtheorem{introtheorem}{Theorem}

\theoremstyle{definition}
\newtheorem{definition}[lemma]{Definition}

\newtheorem{example}[lemma]{Example}

\theoremstyle{remark}
\newtheorem{remark}[lemma]{Remark}

\newtheorem{notation}[lemma]{Notation}

\newcommand{\Hom}{\operatorname{Hom}}
\newcommand{\Inj}{\operatorname{Inj}}

\newcommand{\sgn}{\operatorname{sgn}}
\newcommand{\eps}{\varepsilon}
\newcommand{\triv}{{\mathbbm 1}}

\newcommand{\Ind}{\operatorname{Ind}}
\newcommand{\id}{\operatorname{Id}}

\renewcommand{\Im}{\operatorname{Im}}

\newcommand{\Ext}{\operatorname{Ext}}

\newcommand{\Aut}{{\operatorname{Aut}}}

\newcommand{\End}{\operatorname{End}}

\newcommand{\C}{{\mathbb C}}
\newcommand{\Z}{{\mathbb Z}}

\newcommand{\Del}{{\Delta}}

\newcommand{\lam}{{\lambda}}

\newcommand{\fh}{{\mathfrak{h}}}

\newcommand{\gl}{{\mathfrak{gl}}}
 \newcommand{\Sym}{\operatorname{Sym}}
\newcommand{\abs}[1]{\left|{#1}\right|}

\newcommand{\uRep}{\underline{{\rm Rep}}}

\newcommand{\Rep}{\operatorname{Rep}}
\newcommand{\res}{\operatorname{res}}
\newcommand{\X}{\mathbf{X}}

\newcommand{\V}{{\mathcal V}}
\newcommand{\A}{{\mathcal A}}
\newcommand{\U}{\mathcal{U}}
\newcommand{\co}{{\it O}}
\newcommand{\InnaB}[1]{{{#1}}}
\newcommand{\InnaA}[1]{{{#1}}}

\DeclareMathOperator{\Mod}{Mod}
\newcommand{\tors}{\mathrm{tors}}

\newcommand{\bdel}{\blacktriangle}

\def\quotient#1#2{%
    \raise1ex\hbox{$#1$}\Big/\lower1ex\hbox{$#2$}%
}

\begin{document}
\title{Deligne categories and representations of the infinite symmetric group}
\date{\today}

 \author{Daniel Barter, Inna Entova-Aizenbud, Thorsten Heidersdorf}
\address{}
\address{D. B.: Department of Mathematics, University of Michigan, Ann Arbor}
\curraddr{Mathematical Sciences Institute, Australian National University, Canberra}
\email{danielbarter@gmail.com} 
\address{I. E.: Department of Mathematics, Ben Gurion University, Beer-Sheva, Israel}
\email{inna.entova@gmail.com}
\address{T. H.: Department of Mathematics, Ohio State University, Columbus}
\curraddr{Max Planck Institut f\"ur Mathematik, Bonn}
\email{heidersdorf.thorsten@gmail.com}

\thanks{2010 {\it Mathematics Subject Classification}: 05E05, 18D10, 20C30.}

\begin{abstract}
We establish a connection between two settings of representation stability for the symmetric groups $S_n$ over $\C$. One is the symmetric monoidal category $\Rep(S_{\infty})$ of algebraic representations of the infinite symmetric group $S_{\infty} = \bigcup_n S_n$, related to the theory of {\bf FI}-modules. The other is the family of rigid symmetric monoidal Deligne categories $\uRep(S_t)$, $t \in \C$, together with their abelian versions $\uRep^{ab}(S_t)$, constructed by Comes and Ostrik. 

We show that for any $t \in \C$ the natural functor $\Rep(S_{\infty}) \to \uRep^{ab}(S_t)$ is an exact symmetric faithful monoidal functor, and compute its action on the simple representations of $S_{\infty}$. Considering the highest weight structure on $\uRep^{ab}(S_t)$, we show that the image of any object of $\Rep(S_{\infty})$ has a filtration with standard objects in $\uRep^{ab}(S_t)$.

 As a by-product of the proof, we give answers to the questions posed by P. Deligne concerning the cohomology of some complexes in the Deligne category $\uRep(S_t)$, and their specializations at non-negative integers $n$. 
\end{abstract}

\keywords{}
\maketitle
\setcounter{tocdepth}{3}
\section{Introduction}

Consider the symmetric monoidal categories $\Rep(S_n)$ of finite-dimensional complex representations of symmetric groups $S_n$. It has been known for some time that there is an interesting phenomena of stabilization of representations of $S_n$ as $n \to \infty$: namely, one often encounters sequences $(V_n)_{n \geq 0}$, where each $V_n$ is a finite-dimensional representation of $S_n$, such that the action of $S_n$ on $V_n$ is determined by permutations with small cycles. 


Such ``stable sequences'' are encountered in many contexts: cohomologies of configuration spaces of $n$ distinct ordered points on a connected oriented manifold, spaces of polynomials on rank varieties of $n \times n$ matrices, the cohomology ring of the moduli space of $n$-pointed curves, and more. A framework for studying sequences with such properties has been developed with the theory of finitely generated ${\bf FI}$-modules by Church, Ellenberg, Farb \cite{CEF} and others. In particular they proved that $dim (V_n)$ is eventually polynomial if $(V_n)_{n \geq 0}$ comes from a finitely generated ${\bf FI}$-module.


Another manifestation of the stabilization phenomena is the Murnaghan stabilization theorem. Consider three arbitrary Young diagrams $\lam, \mu, \tau$. For $n>>0$, denote by ${\lam}[n]$ the Young diagram of size $n$ (and corresponding irreducible representation of $S_n$) obtained by adding a top row of size $n-\abs{\lam}$ to $\lam$ (similarly for $\mu, \tau$). It was proven by Murnaghan in \cite{Mu1} that the sequence $\{\dim\left( {\mu}[n] \otimes {\tau}[n] \otimes {\lam}[n] \right)^{S_n} \}_{n>>0}$ stabilizes. The stable value of the sequence is called the {\it stable (reduced) Kronecker coefficient $\bar{g}^{\lam}_{\mu, \tau}$} associated with the triple $(\lam, \mu, \tau)$.

We consider two types of ``limits'' (in a non-precise sense) of $\Rep (S_n)$ as $n \to \infty$. These limits will be $\C$-linear symmetric monoidal categories, and each type expresses a certain ``stabilization'' phenomena occuring in the family $\Rep(S_n)$, $n \geq 0$.

\begin{itemize}
 \item The algebraic representations of the infinite symmetric group $S_{\infty} = \bigcup_{n} S_n$, where each representation of $S_{\infty}$ corresponds to a ``polynomial sequence'' $(V_n)_{n \geq 0}$.
 \item The ``polynomial'' family of Deligne categories $\Rep(S_t)$, $t \in \C$, where the objects of the category $\Rep(S_t)$ can be thought of as ``continuations of polynomial sequences $(V_n)_{n \geq 0}$'' to complex values of $t=n$. 
\end{itemize}

There is a striking resemblance between the two settings: in both settings, one considers symmetric monoidal abelian categories, and both are a natural ``categorical context'' in which stable Kronecker coefficients appear. The goal of this paper is to study the relationship between these two settings and show that they are closely related.

Let us elaborate more on each of these settings.

\subsection{Representations of the infinite symmetric group}
Consider the infinite symmetric group $S_{\infty} = \cup_n  
S_n$ and let $\Rep(S_{\infty})$ be the category of algebraic representations of 
$S_{\infty}$ in the sense of Sam and Snowden (see \cite{SS}). This is a symmetric monoidal (SM) abelian category generated by the object $\fh$, where $\fh$ is the permutation representation of $S_{\infty}$ on $\C^{\infty}$. 

The category $\Rep(S_{\infty})$ comes endowed with left-exact SM functors $$\Gamma_n : \Rep(S_{\infty}) \longrightarrow \Rep(S_n)$$ called the specialization functors.

The isomorphism classes of simple objects in this category are parametrized by all Young diagrams (of arbitrary size). The functors $\Gamma_n$ behave nicely on simple object:
$$\Gamma_n \left(L_\lam \right) = \begin{cases}
                                   \lam[n] &\text{  if  } n \geq \abs{\lam} + \lam_1 \\
				    0 &\text{  else  } 
                                  \end{cases}$$
 (hence $L_{\lambda}$ should be considered as the representation of the symmetric group $S_{\infty}$ corresponding to a Young diagram of infinite size, obtained from $\lambda$ by adding an infinite row on top).

Considering the Grothendieck ring of $\Rep(S_{\infty})$, one sees that the structure constants in this ring are the stable Kronecker coefficients; that is, given three simple objects $L_{\lam}, L_{\mu}, L_{\tau}$, the multiplicity of $L_{\lam}$ in the Jordan-Holder components of $L_{\mu} \otimes L_{\tau}$ is precisely $\bar{g}^{\lam}_{\mu, \tau}$.

The category $\Rep(S_{\infty})$ possesses a nice universal property, as 
described in \cite[Theorem 6.4.3]{SS}: there is an isomorphism between the category of left-exact SM functors of 
$\Rep(S_{\infty})$ to a SM abelian category $\mathcal{C}$ and the category of 
Frobenius objects (see Definition \ref{def:frobenius_object}) in $\mathcal{C}$. In case $\mathcal{C} = \Rep(S_n)$ and the Frobenius object in it is taken to be the permutation representation $\C^n$, the corresponding functor would be canonically isomorphic to $\Gamma_n$.

The category $\Rep(S_{\infty})$ can be constructed in another way. In \cite{SS}, Sam and Snowden check that the  colimit functor ${\bf FI} \mbox{-} \Mod \to \Rep(S_{\infty})$ is exact and essentially surjective. Moreover, an object is sent to zero if and only if it is finite length in ${\bf FI} \mbox{-} \Mod$. This establishes an equivalence of abelian categories between ${\bf FI} \mbox{-} \Mod / {\bf FI}^{\rm finite} \mbox{-} \Mod$ and $\Rep(S_{\infty})$. This equivalence does not preserve the tensor structure.

\subsection{Deligne categories}
We consider the family of categories $\uRep(S_t)$ (parametrized by $t \in \C$) as constructed by Deligne (\cite{Del07}, see also \cite{CO}). These categories are rigid SM Karoubian additive $\C$-linear categories, with $\End(\triv) = \C$. Each category $\uRep(S_t)$ is generated by a single object $\fh_t$ of dimension\footnote{The notions of dimension of an object is defined in any rigid SM category (see \cite{EGNO}), and is by definition an element of $\End(\triv)$. In our case it is just a complex number.} $t$, which plays the role of the ``permutation representation'' for the non-existent group $S_t$. 

The categories $\uRep(S_t)$ ``interpolate'' the categories $\Rep(S_n)$ when $t = n \in \Z_{\geq 0}$ in the following sense: there exists an additive SM functor 
$$\mathcal{S}_n: \uRep(S_n) \longrightarrow \Rep(S_n) $$ which is full and essentially surjective, making $\Rep(S_n) $ a quotient of $\uRep(S_n)$ by a tensor ideal.


The indecomposable objects in $\uRep(S_t)$ are parametrized (up to isomorphism) by the set of all Young diagrams, and the structure constants are described in terms of stable Kronecker coefficients $\bar{g}^{\lambda}_{\mu, \tau}$ (see \cite{En2}). The indecomposable objects interact nicely with the specialization functors: for any $n \in \Z_{\geq 0}$ and any indecomposable object $\X_{\lam} \in \uRep(S_n)$, we have 

$$\mathcal{S}_n \left(\X_{\lam}\right) = \begin{cases}
                                   \lam[n] &\text{  if  } n \geq \abs{\lam} + \lam_1 \\
				    0 &\text{  else  } 
                                  \end{cases}$$

The Deligne category $\uRep(S_t)$ is the universal SM category with a dualizable Frobenius object of dimension $t$ (see \cite{Del07, CO1} for a precise statement). The category  $\uRep(S_t)$ is abelian if and only if $t \notin \Z_{\geq 0}$. Therefore we will also consider the categories $\uRep^{ab}(S_t)$ ($t \in \C$), the abelian envelopes of $\uRep(S_t)$ as constructed by Comes and Ostrik (\cite{CO1}).

These are abelian $\C$-linear rigid SM categories with exact tensor 
product functors, and $\End(\triv) = \C$ (such categories are called tensor categories in \cite{EGNO}). The category $\uRep(S_t)$ is a full SM subcategory in $\uRep^{ab}(S_t)$. In case $t \notin \Z_{\geq 0}$, $\uRep(S_t)$ is isomorphic to its abelian envelope.

For $t=n \in \Z_{\geq 0}$, however, there is a caveat: there is no SM functor from $\uRep^{ab}(S_{t=n})$ to $\Rep(S_n)$. Yet $\uRep^{ab}(S_t)$ has many upsides: not only it is abelian, it is in fact a highest-weight category (with an 
infinite weight poset given by all Young diagrams, see \cite{En}). The subcategory $\uRep(S_t)$ is then the subcategory of tilting objects.

\subsection{Main result} It is a natural question that has been raised in the literature \cite{CEF} \cite{SS} how these two limit constructions are related. The universal  property of $\Rep(S_{\infty})$ yields a left-exact $\C$-linear symmetric monoidal functor ${\mathbf \Gamma}_t: 
\Rep(S_{\infty}) \to \uRep^{ab}(S_t)$. We prove the following result (for a more precise statement see \ref{thrm:main}):

\begin{introtheorem}\label{introthrm:main}
\mbox{}
\begin{itemize}
 \item The functor ${\mathbf \Gamma}_t$ is faithful and exact.
 \item The functor ${\mathbf \Gamma}_t$ takes simple objects in $\Rep(S_{\infty})$ to 
standard objects in the highest-weight category $\uRep^{ab}(S_t)$, and 
injective objects to tilting objects (these are precisely the objects coming 
from the Karoubian Deligne category $\uRep(S_t)$). 
\end{itemize}

\end{introtheorem}

\begin{remark}
 In particular, the essential image of the functor $\mathbf{\Gamma}_t$ consists of standardly-filtered objects. Note that there are standardly filtered objects which are not in the essential image of the functor, see remark \ref{rem:standard}
 \end{remark}

\begin{remark}
 We stress that this functor is not full: for instance
there are no non-zero maps from the trivial representation $\triv$ to the 
permutation representation $\fh_{\infty} = \C^{\infty}$ in $\Rep(S_{\infty})$, but there is a 
non-zero morphism $$\mathbf{\Gamma}_t(\triv) = \triv \to \fh_t = \mathbf{\Gamma}_t(\fh_{\infty})$$ in 
$\uRep^{ab}(S_t)$, an analogue of the morphism $$\C \longrightarrow \C^n, 
\;\; 1 \to \sum_i e_i$$ of $S_n$ representations.
\end{remark}

\begin{remark} \label{rem:not-injective} The functor $\mathbf{\Gamma}_t$ is not injective on objects. An example is given by \[ \mathbf{\Gamma}_t(\fh_{\infty}) = \mathbf{\Gamma}_t( \mathbf{1} \oplus L_{\yng(1)})\] for all $t \geq 2$. 
\end{remark}

\begin{remark}
 A similar result for the general linear group has been proven in \cite{EHS}.
\end{remark}

\subsection{Delta complexes and a conjecture of Deligne}

The crucial point of the proof of Theorem \ref{introthrm:main} is to show the exactness of $\mathbf{\Gamma}_t$. It is very difficult to evaluate the effect $\mathbf{\Gamma}_t$ on objects and morphisms from its definition. In particular we cannot calculate $\mathbf{\Gamma}_t (I^{\bullet}_{\lambda})$ for an injective resolution $I^{\bullet}_{\lambda}$ of $L_{\lambda}$ directly. To circumvent this problem we define special objects $\Delta_k^{\infty}$, $\Delta_k^{t}$, $\Delta_k^{n}$ ($k \in \Z_{\geq 0}$) in each of the categories $\Rep(S_{\infty})$, $\uRep(S_t)$, $\Rep(S_N)$ (for any $N \in \Z_{\geq 0}, t \in \C$). These objects come equipped with morphisms $\res^i: \Delta_k  \to \Delta_{k-1}$ where $\Delta_k$ is any one of $\Delta_k^{\infty}$, $\Delta_k^{t}$, $\Delta_k^{n}$.

The benefit of these objects and morphisms is that we can evaluate the effect of $\mathbf{\Gamma}_t$ on them. By Lemma \ref{lem:Delta_objects_corresp} $\mathbf{\Gamma}_t$ takes $\Delta_k^{\infty}$ to $\Delta_k^t$; and by Lemma \ref{lem:res_morph_corresp} $\mathbf{\Gamma}_t$ takes $\res^i: \Delta_k^{\infty} \to \Delta_{k-1}^{\infty}$ to $\res^i: \Delta_k^t \to \Delta_{k-1}^t$ for any $i \in \{1, \ldots k\}$.

Using the objects $\Delta_k^{\infty}$, $\Delta_k^{t}$, $\Delta_k^{n}$ and the morphisms $\res^i$ as building blocks, we define certain complexes in each of the categories $\Rep(S_{\infty})$, $\Rep(S_N)$, $\uRep(S_t)$ (for any $N \in \Z_{\geq 0}, t \in \C$), which we call {\it Delta complexes}. These come with a natural $S_n$-action, and are denoted by $K_{n, \infty}^{\bullet}$, $K_{n, N}^{\bullet}$, $K_{n, t}^{\bullet}$ respectively.  

In Sections \ref{sec:Delta_complex_infty}, \ref{sec:Delta_complex_Deligne}, we prove the following result:

\begin{introtheorem}\label{introthrm:complex}
\begin{enumerate}
\item The complexes $K_{n, \infty}^{\bullet}$, $K_{n, t}^{\bullet}$ (for any $t \in \C$) have cohomology only in the top degree. 
\item For any $\lambda \vdash n$, $\Hom_{S_n}(\lambda ,  K_{n, \infty}^{\bullet})$ is the injective resolution of $L_{\lambda}$ in $\Rep(S_{\infty})$.
\item The complexes $K_{n, \infty}^{\bullet}$, $K_{n, N}^{\bullet}$, $K_{n, t}^{\bullet}$ correspond to each other under the functors $ \Gamma_N, \mathcal{S}_N,{\mathbf \Gamma}_t$.

\end{enumerate}
\end{introtheorem}

The last two statements imply the exactness of the functor ${\mathbf \Gamma}_t$. Using the first two statements of Theorem \ref{introthrm:complex}, we also compute the cohomology of the complex $K_{n, N}^{\bullet}$ for any $N$, answering the questions posed by P. Deligne in \cite[(7.13)]{Del07}.

\subsection{Corollaries}

Theorem \ref{introthrm:main} unifies the two ``stabilization settings'' described above, and explains the similarities between $\Rep(S_{\infty})$ into $\uRep^{ab}(S_t)$, such as the occurrence of stable Kronecker coefficients in both settings.

Let us give a few examples of additional corollaries and applications of Theorem \ref{introthrm:main}:

\begin{itemize}
\item One of the main tools used so far in working with the Deligne categories was the {\it lifting map}: a ring homomorphism between the Grothendieck rings of $\uRep(S_t)$ and $\uRep(S_T)$ for any $t \in \C$ and $T$ a formal variable, which respected the $\Hom$ bilinear forms on both rings. 
Until now, this has been a crucial tool in passing between different values of $t$, giving a formal meaning to the ``polynomiality'' of the family of categories $\uRep(S_t)$. Theorem \ref{introthrm:main} gives a new, more categorical tool to do the same: passing between different values of $t$ by passing through the image of $\Rep(S_{\infty})$ in $\uRep^{ab}(S_t)$. 
\item Given an abelian SM category with a rigid Frobenius object $X$, the corresponding functor $\Gamma_X$ from $\Rep(S_{\infty})$ will factor through either $\Rep(S_n)$ for some $n$, or through $\uRep^{ab}(S_t)$ for some $t$ (to determine in which of the cases we are, see \cite{CO1} for an explicit criterion on $X$). This allows us to determine whether $\Gamma_X$ is exact (and calculate its derived functors).
\item Each ${\bf FI}$-module, up to torsion, corresponds to an object of the category $\uRep^{ab}(S_t)$ for any $t$. 
\item We use Theorem \ref{introthrm:main} to provide an easy computation of the structure constants in the Grothendieck ring of $\uRep^{ab}(S_t)$ (see Section \ref{sec:Grothendieck}).
\end{itemize}

It would be interesting to study the relation between the corresponding settings in positive characteristic.

\subsection{Acknowledgements}
This project was started during the Representation Stability Workshop at the American Institute of Mathematics (AIM) in San-Jose in 2016. The authors thank the organizers and the Institute for the hospitality and the opportunity to collaborate on this project. 

The second author would like to thank Vladimir Hinich for his valuable remarks and suggestions.

The authors would like to thank the referee for valuable suggestions.

\subsection{Structure of the article}
Section \ref{sec:notation} gives some general preliminaries on representations of symmetric groups and tensor powers of complexes. In Sections \ref{sec:infinity} and \ref{sec:Deligne}, we describe the main players: the categories $\Rep(S_{\infty})$ and $\uRep^{ab}(S_t)$, and prove some auxilary results. In Section \ref{sec:Del_correspondence}, we prove that the functor $\Gamma_t$ takes objects $\Delta_k^{\infty}$ to $\Delta_k^t$. 

In Sections \ref{sec:Delta_complex_infty}, \ref{sec:Delta_complex_classical}, \ref{sec:Delta_complex_Deligne}, we study the Delta complexes and prove Theorem \ref{introthrm:complex}. In Section \ref{sec:proof_main}, we prove Theorem \ref{introthrm:main}.
Finally, in Section \ref{sec:Grothendieck} we use Theorem \ref{introthrm:main} to give an easy computation of the structure constants in the Grothendieck ring of $\uRep^{ab}(S_t)$.


\section{Notation and preliminaries}\label{sec:notation}
The base field throughout the paper will be $\C$.

\subsection{Representations of general linear groups}
\InnaA{Given a finite-dimensional vector space $U$, we denote by $\Rep(GL(U))_{pol} $ the full subcategory of finite-dimensional polynomial representations of the algebraic group $GL(U)$: these are representations $GL(U) \to GL(U')$ which can be extended to a morphism of affine varieties $End(U) \to End(U')$. Equivalently $\Rep(GL(U))_{pol}$ is a semisimple category whose simple objects are $S^{\lam} U$ (application of the Schur functor $S^{\lam}$ to $U$). This is a full subcategory of $\Mod_{\U(\gl(U))}$, where $\gl(U)$ is the Lie algebra of $GL(U)$.}

\subsection{Representations of the symmetric group}\label{ssec:prelim_rep_sym}
A partition $\lambda = (\lambda_1,\lambda_2, \ldots)$ is a non-increasing sequence of integers such that $\lambda_i = 0$ for some $i$. We often identify a partition with its associated Young diagram. The size $|\lambda|$ of a partition $\lambda$ is the sum $\sum \lambda_i$. We write $\lambda\vdash n$ if $\lambda$ has size $n$. Given two partitions $\lambda, \mu$ such that $\lambda_i \leq \mu_i$ for all $i$ we write $\lambda \subseteq \mu$. If $\lambda \subseteq \mu$, $\mu / \lambda$ denotes the diagram that is obtained from $\mu$ if we remove all the boxes in its Young diagram that are part of $\mu$. We say $\mu / \lambda$ is a horizontal strip of size $|\mu| - |\lambda|$ if $\mu_i \geq \lambda_i \geq \mu_{i+1}$ and write $\mu / \lambda \in HS_d$ where $d = |\mu| - |\lambda|$. We write $\mu / \lambda \in HS$ if there is some $d$ for which $\mu / \lambda \in HS_d$. We denote by $\mu^{\vee}$ the transpose partition. We say $\mu / \lambda \in VS_d$ if and only if $\mu^{\vee} /  \lambda^{\vee} \in HS_d$.

We denote by $S_n$ the symmetric group. If $\lambda$ is a partition of $n$, we let $\lambda$ be the corresponding irreducible representation of $S_n$. They are indexed in such a way that $(n)$ is the trivial representation of $S_n$, $(n-1,1)$ is the standard representation and $(1^n)$ (the $n$-th exterior power of the standard representation) is the sign representation $\sgn_n$ of $S_n$.

We will denote by $\Rep(S_n)$ the category of finite-dimensional representations of $S_n$.

Let $N \in \mathbb{Z}_{\geq 0} \cup \{\infty\}$. We use the following notation:
$$[N] := \{ j \in \Z : 1 \leq j \leq N\}$$
For $n \in \mathbb{Z}_{\geq 1}$, and $1\leq i\leq n$, we denote:
 $$\res^i = \iota_i^*:  Fun\left([n],[N]\right) \longrightarrow Fun\left([n-1], [N]\right) $$ 
where $\res^i (f) := f \circ \iota_i$, and \begin{align}\label{eq:iota} \iota_i: [n-1] \to  [n]  , \;\;\; \iota_i(j) = \begin{cases} j &\text{  if   } j<i \\ j+1 
&\text{  if   } j \geq i.\end{cases}\end{align}

We will also denote by $\res^i$ the restriction of the above map to the set $\Inj\left([n],[N]\right)$ of injective maps, and the induced linear map
$$\res^i: \C \Inj\left([n],[N]\right) \to \C \Inj\left([n-1],[N]\right).$$


\subsection{Induced representations of symmetric groups}\label{ssec:prelim_induced}

Throughout the next few sections, we will heavily use the following induced representations of symmetric groups:
$$ \Delta^n_k:= \Ind_{S_{n-k}\times S_k}^{S_n \times S_k} \C, \;\; \bdel^n_k:= \Ind_{S_{n-k}\times S_k}^{S_n \times S_k} \C \otimes \sgn_k$$
Here the induction goes by the inclusions of subgroups $S_{n-k} \times S_k \subset S_{n-k} \times S_k  \times S_k \subset S_n \times S_k$ (the first inclusion being the diagonal $S_k \subset S_k \times S_k$), and $\sgn_k$ is the sign representation of $S_k$.

We will also use the $S_n$-equivariant maps $$\res^i: \Delta^n_k \to \Delta^n_{k-1}, \;\; \res^i: \bdel^n_k \to \bdel^n_{k-1}$$ defined before.

We will need the decomposition of these into irreducible representations:
\begin{lemma}\label{lem:Delta_classical_decomp}
 We have an isomorphism of $S_n \times S_k$-modules 
$$\Delta_k^n \cong \C \Inj([k], [n]) \cong \bigoplus_{\substack{\mu \vdash n, \, \lambda \vdash k, \\ \lambda \subset \mu, \, \mu / \lambda \in HS}} \mu \otimes \lambda$$
and 
$$ \bdel^n_k \cong \C \Inj([k], [n]) \otimes (\C \otimes \sgn_k) \cong \bigoplus_{\substack{\mu \vdash n, \, \lambda \vdash k, \\ \lambda \subset \mu, \, \mu / \lambda \in HS}} \mu \otimes \lambda^{\vee}$$
Equivalently,  
$$ \bdel_{k}^{n} \cong \bigoplus_{\substack{\mu \vdash n, \; \lambda \vdash k, \\ \lambda \subset \mu^{\vee}, \; \mu^{\vee} / \lambda \in VS} } \mu \otimes \lambda$$
\end{lemma}
This statement is probably well-known: the first statement can be found, for example, in \cite[Proposition 6.4]{Del07}. We give a short proof below, for completeness of presentation.
\begin{proof}
 The isomorphisms $$\Delta_k^n \cong \C \Inj([k], [n]), \;\; \bdel^n_k \cong \C \Inj([k], [n]) \otimes (\C \otimes \sgn_k)$$ follow easily from the definition of induction.

Let us now consider the decompositions into irreducible $S_n \times S_k$ representations. We induce step-by-step through the inclusions $S_{n-k} \times S_k \subset S_{n-k} \times S_k  \times S_k \subset S_n \times S_k$:

$$\Ind_{S_{n-k}\times S_k}^{S_{n-k}\times S_k \times S_k} \C = \bigoplus_{\lambda \vdash k} \lambda \otimes \lambda, \;\; \Ind_{S_{n-k}\times S_k}^{S_{n-k}\times S_k \times S_k} \C \otimes \sgn_k = \bigoplus_{\lambda \vdash k} \lambda \otimes (\lambda \otimes \sgn_k) = \bigoplus_{\lambda \vdash k} \lambda \otimes \lambda^{\vee}$$

By the branching rules for symmetric groups, we obtain $$\Ind_{S_{n-k}\times S_k}^{S_n \times S_k} \C = \bigoplus_{\lambda \vdash k} \Ind_{S_{n-k}\times S_k \times S_k}^{S_n \times S_k}( \lambda) \otimes \lambda = \bigoplus_{\substack{\mu \vdash n, \, \lambda \vdash k, \\ \lambda \subset \mu, \, \mu / \lambda \in HS}} \mu \otimes \lambda
$$ and 
\begin{align*}
\Ind_{S_{n-k}\times S_k}^{S_n \times S_k} \C \otimes \sgn_k = \bigoplus_{\lambda \vdash k} \Ind_{S_{n-k}\times S_k \times S_k}^{S_n \times S_k}( \lambda) \otimes \lambda^{\vee} = \bigoplus_{\substack{\mu \vdash n, \, \lambda \vdash k, \\ \lambda \subset \mu, \, \mu / \lambda \in HS}} \mu \otimes \lambda^{\vee} 
\end{align*}
This completes the proof of the lemma.
\end{proof}

We will use these representations to consider tensor powers of graded spaces. Let $V = V_0 \oplus V_1$ be a graded vector space with $\dim V_0 =1$ and a fixed vector $v_0 \in V_0 \setminus \{0\}$. Fix $n \geq 0$.

Then $S_n$ acts on the $\Z$-graded vector space $V^{\otimes n}$ by permuting the tensor factors. This commutes with the action of $GL(V_0) \times GL(V_1)$. 

\begin{lemma}\label{lem:tens_pwr_graded}
We have an isomorphism of $\Z$-graded $(GL(V_1) \times S_n)$-modules 
$$ V^{\otimes n}  \cong \bigoplus_{ k\in \Z} \left(V_1^{\otimes k} \boxtimes \Delta_k^n\right)^{S_k}.$$
This isomorphism is given as follows: for any homogeneous tensor $v = v_{1} \otimes v_{2} \otimes \ldots \otimes v_{n}$, define $I \subset [n]$ such that $v_i =v_0$ iff $i \notin I$ (so $k:=\abs{I}$ is the grade of $v$). Set $f_I: [k] \to [n]$ be the order-preserving inclusion whose image is $I$, so $f_I \in \Inj([k], [n])$.
Then
$$v = v_{1} \otimes v_{2} \otimes \ldots \otimes v_{n} \longmapsto e_{S_k}( \otimes_{i \in I} v_i \otimes f_I) $$
where $e_{S_k}$ is the idempotent of taking $S_k$-invariants in $V_1^{\otimes k} \boxtimes \Delta_k^n$.
\end{lemma}

The proof of this lemma can be found \cite[Lemma A.3.3]{En}. 

 \begin{remark}
Considering the Hilbert series of $V^{\otimes n}$, we obtain the binomial decomposition $$(1+t)^n = \sum_{k \in \Z} t^k \binom{n}{k}.$$
 \end{remark}

Next, consider a super vector space $V = V_0 \oplus V_1$ with a fixed isomorphism $\C \to V_0$, and $n \geq 0$. Again, the super vector space $V^{\otimes n}$ has a $\Z$-grading (compatible with the supergrading) and natural action of $S_n$ on it permuting the tensor factors, and commuting with the action of $GL(V_0) \times GL(V_1)$. But one has to be careful: the $S_n$-action comes from the braiding in the appropriate category of super vector spaces which differs from the braiding on usual vector spaces
$$ v \otimes v' \mapsto (-1)^{deg(v)deg(v')} v' \otimes v$$ where $v,v'$ are homogenuous elements in $V$ and $deg(v)$ denotes the parity of $v$.

\begin{lemma}\label{lem:tens_pwr_super}
We have an isomorphism of $\Z$-graded (super) $(GL(V_1) \times S_n)$-modules 
$$ V^{\otimes n}  \cong \bigoplus_{ k\in \Z} \left(V_1^{\otimes k} \boxtimes \bdel_k^n\right)^{S_k}.$$
Using the same notation as in Lemma \ref{lem:tens_pwr_graded}, the isomorphism is defined as 
$$v = v_{1} \otimes v_{2} \otimes \ldots \otimes v_{n} \longmapsto e_{S_k}( \otimes_{i \in I} v_i \otimes f_I) $$
where $e_{S_k}$ is the idempotent of taking $S_k$-invariants in $V_1^{\otimes k} \boxtimes \bdel_k^n$.
\end{lemma}

The proof of this lemma is analogous to the proof of the previous lemma and is left to the reader. We should note, however, that the statement traces back to \cite{Ser}. In particular, we will use Lemma \ref{lem:tens_pwr_super} to consider tensor powers of length $2$ complexes (cf. \cite[Lemma 7.10]{Del07}). The following result is a direct analogue of Lemma \ref{lem:tens_pwr_super}:

\begin{lemma}\label{lem:tens_pwr_cmplx}
\mbox{}
\begin{enumerate}
 \item  Let $C^{\bullet}:  C_{-1}  \xrightarrow{\partial} C_0$ be a complex of vector spaces, with a fixed isomorphism $C_0 \cong \C$. The complex $(C^{\bullet})^{\otimes n}$ has a natural action of $S_n$ on it: as an $S_n$-module, the $(-k)$-th term of the complex $(C^{\bullet})^{\otimes n}$ is of the form
$$\left( C_{-1}^{\otimes k} \boxtimes \bdel_k^n \right)^{S_k}$$ with differentials $$\sum_{i=1}^k (-1)^i \partial^{(i)} \otimes res^i$$ (here $\partial^{(i)}$ is the differential $\partial$ applied to the $i$-th tensor factor).
 \item Let $C^{\bullet}:  C_0  \xrightarrow{\partial} C_1$ be a complex of vector spaces, with a fixed isomorphism $C_0 \cong \C$. The complex $(C^{\bullet})^{\otimes n}$ has a natural action of $S_n$ on it: as an $S_n$-module, the $k$-th term of the complex $(C^{\bullet})^{\otimes n}$ is of the form
$$\left( C_1^{\otimes k} \boxtimes \bdel_k^n \right)^{S_k}$$ with differentials $$\sum_{i=1}^k (-1)^i \partial^{(i)} \otimes (res^i)^*$$ (here $\partial^{(i)}$ is the differential $\partial$ applied to the $i$-th tensor factor, and $(res^i)^*: \bdel_{k-1}^n \to \bdel_k^n$ is dual to $res^i$).
\end{enumerate}

\end{lemma}

This result easily generalizes to the case when the category of vector spaces is replaced by any Karoubi additive symmetric monoidal category with a fixed isomorphism $C_0 \cong \triv$.



\section{The category \texorpdfstring{$\Rep(S_{\infty})$}{Rep(S)}}\label{sec:infinity}

\subsection{The category \texorpdfstring{$\Mod_A$}{A-Mod}}\label{ssec:cat_A_mod} Let $\V$ denote the category of polynomial representations of $\mathfrak{gl}_{\infty} = \bigcup_n \mathfrak{gl}(n)$, where a representation of $\mathfrak{gl}_{\infty}$ is polynomial if it appears as a direct summand of an arbitrary direct sum of tensor powers of the standard representation $V=\C^{\infty}$. 

This is a semisimple category, and the simple objects are the objects of the form $S_{\lambda}V$ (the Schur functor $S_{\lambda}$ applied to $V$). 

This category is equivalent to the following categories (see \cite[1.2]{SS-2}):

\begin{itemize} 
\item The category of polynomial endofunctors of Vec. The morphisms are the natural transformations. The simple objects are the Schur functors $S_{\lambda}$ where $\lambda$ is a partition of arbitrary size.
\item The category of sequences $(V_n)_{n \geq 0}$ where $V_n$ is a representation of $S_n$. The morphisms $f: (V_n) \to (W_n)$ are sequences of maps of representations $f = (f_n), f_n: V_n \to W_n$. The simple objects are the representations $\lambda$ placed in degree $|\lambda|$ with $0$ in all other degrees.
\end{itemize}

For an explicit description of these categories and equivalences see \cite[1.2]{SS-2} \cite[5.1 - 5.4]{SS-3}. Here we recall only the following facts. The category $\V$ possesses a natural grading $\V = \bigoplus_n \V^n$ where $\V^n$ is an additive category containing $V^{\otimes n}$. In particular, $\V^n$ contains $S^{\lambda} V$ for any $\lambda \vdash n$. The equivalence between $\V$ and the category of sequences $(V_n)_{n \geq 0}$ is given as follows:

\begin{lemma}
 There is an equivalence of abelian categories 
$${\Phi}: \V \xrightarrow{\sim} \bigoplus_n \Rep(S_n), \;\;\; W \in \V^n \longmapsto \Hom_{\V}(W, V^{\otimes n})$$
\end{lemma}

By Schur-Weyl duality, we have, for any partition $\lambda \vdash n$: $$\Phi(S^{\lambda} V) = \lambda$$ where $\lambda$ stands for the irreducible representation of $S_n$ associated to the corresponding partition. It is important to stress that $\Phi$ is {\it not} a monoidal equivalence: the tensor product in $\V$ corresponds to 
induction operation in $\bigoplus_n \Rep(S_n)$. 

\begin{lemma} \cite[6.1.4]{SS-3}
 We have an isomorphism of $S_n$-representations $$\Phi(V^{\otimes n}) \cong 
\mathbb{C}[S_n]$$ where $\mathbb{C}[S_n]$ is the right regular representation of $S_n$ placed in degree $n$.
\end{lemma}

We denote by $A$ the twisted commutative algebra \cite[1.3]{SS-2} \cite{SS-3} $A = \bigoplus_{d \geq 0} \Sym^d(\C^{\infty})$ It is a commutative algebra object in $\V$. An $A$-module is by definition an object in $\V$ with a multiplication map $A \otimes M \to M$. An $A$-module $M$ is finitely generated if there exists a surjection $A \otimes V \to M \to 0$ where $V$ is a finite length object of $\V$. $M$ has finite length as an $A$-module if and only if it has finite length as an object of $\V$ (i.e. it is a finite direct sum of simple objects). We denote by $\Mod_A$ the category of finitely generated $A$-modules. It is an abelian category which is equivalent to the category ${\bf FI} \mbox{-} \Mod$ of $FI$-modules \cite{CEF}.

\begin{lemma} \cite[Proposition 1.3.5] {SS-2} \label{thm:fi-mod}
 The functor $\Phi$ induces an equivalence $$\widehat{\Phi}: \Mod_A \xrightarrow{\sim} {\bf FI} \mbox{-} \Mod$$ of abelian categories.
\end{lemma}

\begin{remark}
 This is {\it not} an equivalence of monoidal categories.
\end{remark}



\subsection{The category \texorpdfstring{$\Mod_K$}{Mod(K)}}\label{ssec:Mod_K} We denote by $\Mod_K$ the Serre quotient of $\Mod_A$ by the Serre subcategory of finite length objects in $\Mod_A$ and we let $T: \Mod_A \to \Mod_K$ be the localization functor. We define $Q_{\lambda} = T(A \otimes S_{\lambda})$. It has a simple socle which we denote by $L_{\lambda}$. Every simple object is of this form, and $Q_{\lambda}$ is the injective hull of $L_{\lambda}$ \cite[Corollary 2.2.7, Corollary 2.2.14, Proposition 2.2.8] {SS-2}. In particular the simple objects in $\Mod_K$ are parametrized by Young diagrams of arbitrary size. 


\subsection{The category \texorpdfstring{$\Rep(S_{\infty})$}{Rep(S)}}\label{ssec:cat_S_infty}

 Let $[\infty] = \{1,2,\ldots\}$ and denote by $S_{\infty}$ the group of permutations of $[\infty]$ which fix all but finitely many elements. Then $S_{\infty}$ acts on $\fh = \C^{\infty}$ by permuting the basis vectors. A representation of $S_{\infty}$ is algebraic if it appears as a subquotient of a direct sum of the representations $\fh^{\otimes n}$. We denote by $\Rep(S_{\infty})$ the abelian category of algebraic representations. The categories $\Rep(S_{\infty})$ and $\Mod_K$ are equivalent as abelian categories by \cite[Theorem 6.2.4]{SS}. We denote its simple objects again by $L_{\lambda}$ and the injective hulls by $Q_{\lambda}$. Note that they are not equivalent as tensor categories with their obvious tensor product as discussed in \cite[6.2.7]{SS}. However both tensor products are exact. For $\Rep(S_{\infty})$ this can be proven as for $\Mod_K$ in \cite[Remark 2.4.6]{SS}. We will also use the following facts (see \cite{SS} \cite{SS-2}) which follow from the identification with $\Mod_K$.

\begin{proposition}\label{prop:infty_injective_obj}
The indecomposable injective objects $Q_{\lambda}$ in the category 
$\Rep(S_{\infty})$ satisfy:
 \begin{itemize}
 \item The object $Q_\lambda$ has finite length, and its simple constituents 
are those $L_{\mu}$ such that $\lambda/\mu \in HS$ \cite[Corollary 2.2.5]{SS-2}.
 \item Given partitions $\lambda, \mu$, we have \cite[Proposition 2.2.13]{SS-2}
\[
\Hom(Q_\lambda, Q_\mu) \cong \begin{cases} k & \text{if $\lambda / \mu \in 
HS$}\\
0 & \text{otherwise} \end{cases}.
\]
\item Let $\lambda$ be a partition. Define $\mathbb{I}^j = \bigoplus_{\mu,\, 
\lambda/\mu \in VS_j} Q_\mu$. There are morphisms $\mathbb{I}^j \to 
\mathbb{I}^{j+1}$ so that $L_\lambda \to \mathbb{I}^\bullet$ is an injective 
resolution. In particular, the injective dimension of $L_\lambda$ is 
$\ell_{\lambda}$. \cite[Theorem 2.3.1]{SS-2}
\item $Q_{\lambda}$ appears as a summand of $\fh^{\otimes \abs{\lam}}$,
\end{itemize}
\end{proposition}

The category $\Rep(S_{\infty})$ is ``generated'' by the permutation 
representation $\fh_{\infty}=\C^{\infty}$ in the following sense: the category 
$\Rep(S_{\infty})$ has enough injective objects, and the injective objects are 
precisely the direct summands of finite direct sums of $\fh^{\otimes r}$ for 
different $r$.

Let us say a few words about morphisms between tensor powers of $\fh$.

Recall from \cite[(6.3.8)]{SS} that $\Hom_{S_{\infty}}(\mathfrak{h}_{\infty}^{\otimes r}, \mathfrak{h}_{\infty}^{\otimes s})$ has a basis parametrized by equivalence relations on the set $[r] \sqcup [s]$ (i.e. partitions into disjoint subsets), such that each equivalence class meets $[r]$ at least once. Let us give a short explanation of how this parametrization works. Consider a basis of $\mathfrak{h}_{\infty}^{\otimes r}$ parametrized by the set $Fun([r], [\infty])$ of functions from $[r]$ to the set of positive integers. Now, given an equivalence relation $\pi$ on $[r] \sqcup [s]$ as above, let $\pi_r$ be the induced relation on $[r]$ (given by intersecting each equivalence class of $\pi$ with $[r]$). Consider the corresponding map $\psi_{\pi}: \C Fun([r], [\infty]) \to \C Fun([s], [\infty])$. It sends a map $f:[r] \to [\infty]$ to a map $g:[s] \to [\infty]$ if $f$ is constant on each equivalence class of $\pi_r$, and to zero otherwise. The map $g$ is given by $g(i) = f(l)$ where $i \in [s], l \in [r]$  belong to the same part of $\pi$. Composition of such maps $\psi_{\pi}: \C Fun([r], [\infty]) \to \C Fun([r'], [\infty])$, $ \psi_{\pi'}: \C Fun([r'], [\infty]) \to \C Fun([r''], [\infty])$ is then given by $$\psi_{\pi''} = \psi_{\pi'} \circ \psi_{\pi}: \C Fun([r], [\infty]) \to \C Fun([r''], [\infty]) $$ where $\pi''$ is the restriction to $[r] \sqcup [r'']$ of the equivalence relation $\pi \star \pi'$ on $[r] \sqcup [r'] \sqcup [r'']$ induced by $\pi, \pi'$. The relation $\pi \star \pi'$ is the minimal equivalence on $[r] \sqcup [r'] \sqcup [r'']$ relation which is coarser than both $\pi, \pi'$.

For more explanations on a similar parametrization in the finite case (when $[\infty]$ is replaced by $[n]$), see also \cite[Section 2]{CO}.

\begin{notation}
 We will denote by $\Rep_0(S_{\infty})$ the full subcategory whose objects are 
isomorphic to $\fh_{\infty}^{\otimes r}$.
\end{notation}
Then the Karoubi additive envelope of $\Rep_0(S_{\infty})$ (completion with 
respect to finite direct sums and images of idempotents) is the subcategory of 
injective objects of $\Rep(S_{\infty})$.


\subsection{The universal property of \texorpdfstring{$\Rep(S_{\infty})$}{Rep(S)}}\label{ssec:univ_prop_S_infty}

Let $\A$ be a symmetric monoidal category.

\begin{definition}\label{def:frobenius_object}
 A Frobenius object in $\A$ is a tuple $(A, m, \Delta, \eta)$ where $A$ is an object of $\A$, and $m : A \otimes A \to A$ and $\Delta : A \to A \otimes A$ and $\eta : A \to \C$ are morphisms in $\A$, such that:
\begin{itemize}
\item $m$ defines an associative commutative algebra structure on $A$.
\item $(\Delta, \eta)$ defines a counital coassociative cocommutative coalgebra structure on $A$.
\item $m\Delta=\id$ and $\Delta m=(m \otimes 1)(1 \otimes \Delta)$.
\end{itemize}
\end{definition}

We will denote by $T(\A)$ be the category of Frobenius objects in the category $\A$.

The object $\fh \in \Rep(S_{\infty})$ naturally has the structure of an object of $T(\A)$.  The counit $\fh \to \C$ is the augmentation map. The comultiplication $\fh \to \fh \otimes \fh$ sends the basis vector $e_i$ to $e_i \otimes e_i$, while the multiplication $\fh \otimes \fh \to \fh$ sends $e_i \otimes e_j$ to $\delta_{i,j} e_i$ \cite[6.4.2]{SS}.

\begin{theorem} \cite[Theorem 6.4.3]{SS} To give a left-exact tensor functor from $\Rep(S_{\infty})$ to an arbitrary tensor category $\A$ is the same as to give an object of $T(\A)$.
\end{theorem}

Examples of such functors are the specialization functors $$\Gamma_n: \Rep(S_{\infty}) \to \Rep(S_n), \;\; \fh \to \C^n$$ where the generator $\fh$ is sent to the permutation representation of $S_n$. For any simple object $L_{\lambda}$, $\Gamma_n(L_{\lambda}) = \lam[n]$ if $\abs{\lam} + \lambda_1  \leq n$, and $\Gamma_n(L_{\lambda}) = 0$ otherwise. Here $\lam[n]$ is the Young diagram obtained from $\lam$ by adding a top row of size $n- \abs{\lam}$.

The functors $\Gamma_n$ are not exact, and their right derived functors can be explicitly described (see Section \ref{ssec:derived}). In particular, $R^i\Gamma_n(L_{\lambda}) \neq 0$ for some $i$ precisely when $\abs{\lam} + \lambda_1  > n$.


\subsection{The objects \texorpdfstring{$\Delta_k^{\infty}$}{Delta}}\label{ssec:Delta_infty}

In $\Rep(S_{\infty})$ we have - following the notation of Section \ref{ssec:prelim_induced} - the objects \[ \Delta_k^{\infty} = \mathbb{C} Inj\left([k],[\infty]\right).\] Then $\Delta_k^{\infty}$ is a direct summand of $\fh^{\otimes k}= \mathbb{C} Fun \left([k],[\infty]\right)$ and therefore injective.

A basis for the space $\Hom_{S_{\infty}}(\Delta_k^{\infty}, \Delta_{k-1}^{\infty})$ is given by the maps \[ \res^i = \Delta_k^{\infty} \to \Delta_{k-1}^{\infty}.\] described in Section \ref{ssec:prelim_rep_sym}. More generally, we have an isomorphism of $S_k \times S_m$-modules $$\Hom_{S_{\infty}}(\Delta_k^{\infty}, \Delta_{m}^{\infty}) \cong \mathbb{C} Inj\left([m],[k]\right) \cong \Delta_m^k$$ (cf. \cite[Section 6.3]{SS}).

Recall that the object $\Delta_k^{\infty}$ carries an obvious action of $S_k \times S_{\infty}$ (action on $[k]$ and $[\infty]$ respectively). We now give the decomposition of $\Del^{\infty}_k$ into a direct sum of indecomposable injectives.

\begin{lemma}\label{lem:Delta_decomposition_infty}
 As an $S_k \times S_{\infty}$-representation, $\Del^{\infty}_k = 
\bigoplus_{\lambda \vdash k} \lambda \otimes Q_{\lambda}$
\end{lemma}

\begin{proof}
The quotient functor from ${\bf FI}$-modules to $S_{\infty}$-modules sends 
$A \otimes \mathbb{C}[S_k]$ to $\Delta^{\infty}_k$. Then $$ A \otimes \mathbb{C}[S_k] = \bigoplus_{\lambda \vdash k} (A 
\otimes \lambda) \otimes \lambda.$$ Applying the quotient functor ${\bf 
FI} \mbox{-} \Mod \to \Rep(S_{\infty})$ gives $ \Delta^{\infty}_k = \bigoplus_{\lambda \vdash k} Q_{\lambda} \otimes 
\lambda$, as required.
\end{proof}



\section{Deligne categories}\label{sec:Deligne}

\subsection{General description}\label{ssec:Deligne_general}
For $t\in \C$ Deligne defined in \cite{Del07} a rigid symmetric monoidal Karoubi category 
$\uRep(S_t)$ 
with a distinguished object $\mathfrak{h}_t$ which is universal (see \cite{Del07, CO1} for a precise statement).


When $t=d\in \Z_{\geq 0}$ and $\mathfrak{h}_d$ the $d$-dimensional standard representation of 
$S_d$, we obtain a full symmetric monoidal functor $\mathcal{S}_d:\uRep(S_d)\to \Rep(S_d)$ taking $\mathfrak{h}_{t=d}$ to $\mathfrak{h}_d$. This functor is given as the quotient of $\uRep(S_d)$ by the tensor ideal of negligible 
morphisms. The category $\uRep(S_t)$ is a Karoubian category; it is abelian iff 
$t \notin \Z_{\geq 0}$. For $t \in \Z_{\geq 0}$ $\uRep(S_t)$ embeds as a full 
subcategory into an abelian rigid SM category satisfying a universal 
property (\cite{CO1}, \cite{Del07}); this category is called the abelian envelope of $\uRep(S_t)$, and denoted by 
$\uRep^{ab}(S_t)$. We describe its abelian structure in Section \ref{ssec:Deligne_abelian_struct}.

We recall very briefly the construction of $\uRep(S_t)$ (\cite{CO1}, 
\cite{Del07}). The objects in the skeletal Deligne category $\uRep_0(S_t)$ are 
indexed by the non-negative integers. The corresponding object is denoted 
$[n]$. The morphism space $\Hom([n],[m])$ has a basis $
P_{n,m}$ parametrized by partitions of the set $[n] \sqcup [m]$, with the composition of two morphisms $\phi \in P_{n,m}$, $\psi \in P_{m, k}$ given by the induced partition of $[n] \sqcup [k]$ multiplied by a power of $t$. Taking disjoint unions of sets endows $\uRep_0(S_t)$ with the 
structure of a rigid SM category with unit object $[\varnothing]$. The category 
$\uRep_0(S_t)$ has a distinguished object denoted $\fh_t$ of dimension $t$ 
corresponding to the set $[1]$. This object is the analogue of the permutation 
representation of a symmetric group. We define $\uRep(S_t)$ as the idempotent 
completion of the additive envelope of $\uRep_0(S_t)$.

The isomorphism classes of indecomposable elements in $\uRep(S_t)$ are 
parametrized by partitions (or Young diagrams of arbitrary size). For a 
partition $\lambda$ we denote by $\X_{\lambda}$ the corresponding indecomposable 
element. 
\begin{example}
The unit object $\X_{\emptyset} = \triv$ is indecomposable. The object $\fh_t$ 
is decomposable for $t \neq 0$, and decomposes into $$\fh_t = \X_{\yng(1)} 
\oplus \X_{\emptyset}$$ When $t= 0$, the object $\mathfrak{h}_0$ is indecomposable: $\mathfrak{h}_0 = 
\X_{\yng(1)}$.

The object $\X_{\yng(1)}$ is the analogue of the $S_n$-reflection representation 
$\C\{e_i - e_j\}_{1 \leq i \neq j \leq n}$, and has dimension $t-1$ for $t 
\neq 0$, and dimension $0$ for $t=0$.
\end{example}

This is an example of the more general philosophy, which allows one to 
intuitively treat the indecomposable objects of $\uRep(S_t)$ as if 
they were parametrized by ``Young diagrams with a very long top row''. The 
indecomposable object $\X_{\lambda}$ is treated as if it would correspond to 
$\lambda[t]$, i.e. a Young diagram obtained by adding a very long top 
row (``of size $t -\abs{\lambda}$''). This point of view is useful to 
understand how to extend constructions for $S_n$ involving Young diagrams to 
$\uRep(S_t)$.

\begin{example}
\begin{enumerate}
 \item The object $\X_{\emptyset} = \triv$ corresponds to a ``very long top row 
of length $t$'', similarly to the row diagram corresponding to the trivial 
representation of $S_n$. 
\item The object $\X_{\yng(1)}$ corresponds to a ``hook of size $t$ with leg 
$1$'', similarly to the hook diagram corresponding to the reflection 
representation of $S_n$.
\end{enumerate}

\end{example}
\subsection{Specialization functor}\label{ssec:spec_functor_Deligne}
The above intuitive approach has a formal base: whenever 
$\lambda_1+\abs{\lambda} \leq d$, the image of $\X_{\lambda}$ in $\Rep(S_d)$ 
under the functor $\mathcal{S}_d$ is the irreducible representation of $S_d$ 
corresponding to the Young diagram $\lambda[d]$. The rest of the 
indecomposable objects are sent to zero by the functor $\mathcal{S}_d$.

 %



\subsection{The abelian envelope \texorpdfstring{$\uRep^{ab}(S_t)$}{of the Deligne category}}\label{ssec:abelian_envelope}

The category $\uRep(S_t)$ is not abelian for $t \in \Z_+$. However, Deligne constructed an abelian tensor category $\uRep^{ab}(S_t)$ with an embedding $\uRep(S_t) \to \uRep^{ab}(S_t)$. Comes and Ostrik \cite[Theorem 1.2]{CO1} proved the following universal property:


%
%

\begin{theorem}\label{thrm:abelian_envelope}
{\em (cf. \cite[8.21.2]{Del07})} Let $\mathcal{C}$ be a pre-Tannakian category and $T\in \mathcal{C}$ be a dualisable Frobenius object of dimension $d \in \Z_+$. Let $F_T: \uRep(S_t) \to \mathcal{C}$ denote the corresponding tensor functor mapping $\mathfrak{h}_t$ to $T$. Then $F_T$ factors through one of the following:

{\em (a)} The functor $\uRep(S_{d})\to \Rep(S_d)$.

{\em (b)} The functor $\uRep(S_t) \to \uRep^{ab}(S_t)$ from above.
\end{theorem}


\subsection{Blocks and abelian structure}\label{ssec:Deligne_abelian_struct}
As it was said before, the category $\uRep(S_t)$ is semisimple for $t \notin \Z_{\geq 0}$. We will now describe the abelian structure of $\uRep^{ab}(S_t)$ when $t \in \Z_{\geq 0}$.

Fix $t \in \Z_{\geq 0}$. The abelian category $\uRep^{ab}(S_t)$ splits into infinitely many semisimple blocks and finitely many non-semisimple blocks, the latter indexed by the number of Young diagrams of 
size $t$ 
(see 
\cite{CO1}). 

We give a short description of the highest weight structure of $\uRep^{ab}(S_t)$ following \cite{CO1, En}. The partially ordered set of weights is (\{Young diagrams\}, $\geq$) where $\lambda \geq \mu$ iff $\lambda = \lambda^{(i)}$ and $\mu = \lambda^{(j)}$ for some $i,j$ and $i \leq j$.

The non-semisimple blocks are parametrized by partitions $\lambda$ such that $\lambda[t]$ is a Young diagram. Given a non-trivial block corresponding to such a partition $\lambda$, the irreducible objects are parametrized by partitions $\lambda^{(i)}$, $i = \{0,1,\ldots\}$ \[ \lambda = \lambda^{(0)} \subset \lambda^{(1)} \subset \lambda^{(2)} \subset \ldots \] where 
\begin{align}\label{eq:block_Deligne}
 &\lambda^{(i+1)} / \lambda^{(i)} = \text{ strip in row } i+1 \text{ of length } \lambda_i - \lambda_{i+1} + 1 \text{ for } i>0 \\ & \lambda^{(1)} / \lambda^{(0)} = \text{ strip in row } 1 \text{ of length } t - |\lambda| - \lambda_1 + 1. 
\end{align} We say two partitions $\tau$, $\mu$ are equivalent $\tau \sim \mu$ iff $\tau = \lambda^{(i)}$ and $\mu = \lambda^{(j)}$ for some $\lambda \vdash t$ and $i,j \in \{0,1,\ldots\}$. 

\begin{example} \cite[Example 5.10]{CO} For $t = 3$ we have 3 partitions of size $3$: $(3), (2,1)$ and $(1,1,1)$. These are of the form $\lambda[3]$ for the 3 partitions $(0), (1)$ and $(1,1)$. Hence the irreducible elements $\mathbf{L}_{\lambda^{(0)}}$ in the 3 non-semisimple blocks of $\uRep^{ab}(S_3)$ are $\mathbf{L}_{(0)}$, $\mathbf{L}_{(1)}$ and $\mathbf{L}_{(1,1)}$. 
\end{example}

Recall from \cite{CPS} that we have three classes of distinguished objects in a highest weight category: the irreducible objects, the standard objects and the projective covers of the irreducible objects, each parametrized by the elements of the weight poset. In our case the weight poset is given by all Young diagrams, and we denote the irreducible module attached to $\lambda$ by $\mathbf{L}_{\lambda}$, the standard module by $\mathbf{M}_{\lambda}$ and the projective cover by $\mathbf{P}_{\lambda}$ (which is also the injective hull of $\mathbf{L}_{\lambda}$). 

The subcategory $\uRep(S_t) \subset \uRep^{ab}(S_t)$ is the full subcategory of tilting objects (an object is called tilting if it admits a filtration by standard and costandard object). 

%


\subsection{Objects \texorpdfstring{$\Delta^t_k$}{Delta} in the Deligne category}\label{ssec:Delta_obj_Deligne}
In this subsection we define the objects $\Delta^t_k$ in the category $\uRep(S_t)$, and list some of their properties. These objects are defined for any $k \in \Z_+$ and any $t \in \C$, and come with an action of the group $S_k$ on them.

By definition, $\Delta^t_k$ is the image of an idempotent $x_k \in \End_{\uRep(S_t)}(\mathfrak{h}_t^{\otimes k})$ (the latter is given explicitly in \cite[Section 2.1]{CO} and Section \ref{ssec:Deligne_SW}), and satisfies:

\begin{lemma}\label{lem:funct_F_n_Delta_k}
 We have an isomorphism of $S_k$-modules $\mathcal{S}_n(\Delta^t_k) \cong \Delta^n_k$.
\end{lemma}
This is part of the definition of the functor $\mathcal{S}_n$ in \cite[Theorem 6.2]{Del07}, and follows from the fact that $\mathcal{S}_n( \mathfrak{h}_t^{\otimes k}  ) \cong \C Fun([k] , [n])$.

\begin{example}
 $\Delta^t_0 \cong \mathbf{1}$ (unit object in monoidal category $\uRep(S_t)$), $\Delta^t_1 \cong \mathfrak{h}_t$.
\end{example}

\begin{remark}
The objects $\{\Delta^t_k\}_{k\geq 0}$ generate the category $\uRep(S_t)$ as a Karoubian additive category (similarly to $\{\mathfrak{h}_t^{\otimes k} \}_{k \geq 0}$).
%
%
%
\end{remark}

We now describe the $\Hom$-spaces between the objects $\Delta^t_k$. We start by introducing the following notation (cf. \cite[Section 4.3]{En}):

\begin{notation}\label{notn:Delta_maps_Deligne}
\mbox{}
\begin{itemize}[leftmargin=*]
\item We define a {\it partial pairing} $\pi$ of a pair of sets $S, R$ to be a collection $\{ \pi_i \}_{i \in I}, \pi_i \subset S \sqcup R$, such that $\pi_i \cap \pi_j = \emptyset$ if $i \neq j$,$\abs{\pi_i \cap S}, \abs{\pi_i \cap R} \leq 1$ for any $i$, and $\bigcup_{i \in I} \pi_i =S \sqcup R$. The subsets $\pi_i$ will be called {\it parts} of $\pi$. The number of parts of $\pi$ will be denoted by $l(\pi)$. If $r \in R, s\in S$ lie in the same part of $\pi$, we say that $\pi$ {\it pairs} $r$ and $s$.

 \item We denote by $\bar{P}_{r,s}$ the set of partial pairings of the sets $[r], [s]$. In other words, $\bar{P}_{r,s}$ is the set of all equivalence relations $\pi$ on the set $[r] \sqcup [s]$ such that $i, j$ do not lie in the equivalence class whenever $i \neq j \in [r]$, and similarly for $i' \neq j' \in [s]$. 
 \item The following diagrammatic notation will be used for elements of $\bar{P}_{r,s}$: let $\pi \in \bar{P}_{r,s}$. We will represent $\pi$ by a bipartite graph with two rows of aligned vertices: the top row contains $r$ vertices labeled by numbers $1,\ldots,r$, and the bottom row contains $s$ vertices labeled by numbers $1',\ldots,s'$, and there is an edge between $i, j'$ if $i, j'$ are paired by $\pi$.

\end{itemize}
\end{notation}

The following statement describes $\Hom$-spaces between objects $\Delta_r^t, \Delta_s^t$ (cf. \cite[Definition 3.12]{Del07}).

\begin{lemma}\label{lem:Delta_k_homs}
 Let $r, s \geq 1$. The space $\Hom_{\uRep(S_t)} (\Delta^t_r, \Delta^t_s) $ is $\C \bar{P}_{r,s}$, and the composition of morphisms between the objects $\Delta^t_k$ is given as follows: for $\pi \in \bar{P}_{r,s}, \rho \in \bar{P}_{s,t}$, consider the induced partial pairing $ \rho \star \pi \in  \bar{P}_{r,t}$. Then  $\rho \circ \pi$ is a linear combination of partitions $\tau$ which refine $ \rho \star \pi$, and the coefficients are polynomials in $t$ with integer coefficients.

\end{lemma}

We refer to \cite[Section 4.3]{En} for detailed examples of compositions of two morphisms between $\Delta$-objects.

The following special morphisms between the objects $\Delta^t_k$ will be used in this paper.

Let $r\geq 0$, $k \geq 1$, $ 1 \leq l \leq k$.
\begin{definition}\label{def:res_morphisms_Deligne}
Denote by $\res^l$ the morphism $ \Delta^t_{k+1} \rightarrow \Delta^t_{k}$ given by the diagram
$$  \xymatrix{ 1 \ar@{-}[d] & 2 \ar@{-}[d] & 3 \ar@{-}[d] & \ldots & l-1 \ar@{-}[d] & l  & l+1 \ar@{-}[dl] & l+2 \ar@{-}[dl] & \ldots & k+1 \ar@{-}[dl] \\ 1 & 2 & 3 & \ldots & l-1 & l & l+1 & \ldots & k }$$

%
\end{definition}

\begin{remark}\label{rmrk:res_l_image_under_S_n_functor}
Let $n \in \Z_+$. Fix $k, l$ such that $1 \leq k \leq n-1,  1 \leq l \leq k$. Then $$\res^l = \mathcal{S}_n(\res^l): \, \Delta_k^n \rightarrow \Delta_{k-1}^n.$$
%
%
\end{remark}

%

Finally, we state a result which will be crucial in the proof of Theorem \ref{thrm:main}:
\begin{lemma}\label{lem:Delta_t_decomp}
 As an $S_k$-module in $\uRep(S_t)$, the object $\Delta_k^t$ decomposes as 
 $$\Delta_k^{t} \cong \bigoplus_{\mu \vdash k} \bigoplus_{\tau \in B^t_{\mu}}{\mu} \otimes  \X_\tau $$
 where $$B^t_{\mu} = \{ \tau \subset\mu \rvert \mu/ \tau \in HS, \; \text{ and 
either } \; \X_{\tau} \text{ is in a semisimple block, or } \tau = \tau'^{(i)}, \; 
\tau'^{(i+1)} \not\subset \mu\}$$ 
\end{lemma}
\begin{remark}\label{rmk:Delta_t_decomp}
 By \cite[Lemma 4.2.2]{En}, in each non-semisimple block of $\uRep^{ab}(S_{t})$ corresponding to a partition $\tau' \vdash t$, there exist at most two elements $\tau$ of the sequence of Young diagrams $\{\tau'^{(i)}\}_{i \geq 0}$ which satisfy: $ \tau \subset\mu$, $\mu/ \tau \in HS$. Furthermore, if there are two such elements, they are necessarily consecutive. 

Consider the equivalence relation of the set of Young diagrams given by the subdivision of $\uRep^{ab}(S_{t})$ into blocks (called the $\stackrel{t}{\sim}$ equivalence relation in \cite{CO, En}). Then $B^t_{\mu}$ contains at most one element from each equivalence class.
\end{remark}

\begin{proof}
 The statement of the Lemma is a direct consequence of the above remark and \cite[Lemma 7.0.10]{En} (cf. also \cite[Prop. 7.7]{Del07}).
\end{proof}

The object $\Delta_k^t$ is a tilting and thus standardly filtered object in $\uRep(S_t)$, so it also makes sense to understand its standard subquotients. We will use the following lemma:
\begin{lemma}\label{lem:standard-components}
 Let $\tau \in B^t_{\mu}$. Then the standard components of $\mathbf{X}_{\tau}$ are $\mathbf{M}_{\lam}$ such that $\lam\subset \mu$, $\mu /\lam \in HS$ and $\mathbf{L}_{\tau}, \mathbf{L}_{\lam}$ lie in the same block.
\end{lemma}

\begin{proof}
  If $\X_{\tau}$ sits in a semisimple block, then $\tau \in B^t_{\mu}$ implies that $\tau\subset \mu$, $\mu /\tau \in HS$. On the other hand, $\X_{\tau} = \mathbf{M}_{\tau}$, so the statement of the lemma holds.

If $\X_{\tau}$ sits in a non-semisimple block corresponding to the sequence of Young diagrams $\{\tau'^{(i)}\}_{i \geq 0}$ ($\tau' \vdash t$), with $\tau = \tau'^{(i)}$, then (by definition) $\tau \in B^t_{\mu}$ implies that $\tau'^{(i)}\subset \mu$, $\mu /\tau'^{(i)} \in HS$ and $\tau'^{(i+1)} \not\subset \mu$.
Recall from Section \ref{ssec:Deligne_abelian_struct}, \eqref{eq:block_Deligne} that $\tau'^{(i)} / \tau'^{(i-1)}$ is a horizontal strip concentrated in the single row $i$, while $\tau'^{(i)} /\tau'^{(i-2)} \notin HS$. Thus $\mu /\tau'^{(j)} \notin HS$ for $j \neq i, i-1$, and $\mu /\tau'^{(i)} \in HS$. It remains to check whether $\mu /\tau'^{(i-1)} \in HS$. Indeed, $\mu /\tau'^{(i-1)} = \mu /\tau'^{(i)} \sqcup \tau'^{(i)} / \tau'^{(i-1)}$, so to check that this is a horizontal strip, we only need to check whether $\mu /\tau'^{(i)}$ (and hence $\mu$) has any boxes in row $i+1$, column number $\tau'^{(i-1)}_i +1$. That is, we need to check whether $\mu_{i+1} < \tau'^{(i-1)}_i +1$.

Indeed, $\tau'^{(i+1)} \not\subset \mu$, while $\tau'^{(i)}\subset \mu$. Since $\tau'^{(i)}, \tau'^{(i+1)}$ differ only by several boxes in row $i+1$, we conclude that 
$$\mu_{i+1} < \tau'^{(i+1)}_{i+1} = \tau'^{(i-1)}_i +1$$ (the last equality follows from \eqref{eq:block_Deligne}).

Hence  $\tau \in B^t_{\mu}$ implies that the conditions $\tau'^{(j)}\subset \mu$, $\mu /\tau'^{(j)} \in HS$ hold iff $j \in \{i-1, i\}$.

On the other hand, $\X_{\tau'^{(i)}}$ has standard components $ \mathbf{M}_{\tau'^{(i)}}, \mathbf{M}_{\tau'^{(i-1)}}$ if $i >1$ and $\mathbf{M}_{\tau'^{(0)}}$ if $i=0$ \cite[Proposition 4.4.10]{En}. Hence the statement of the lemma holds.
\end{proof}
The following is a direct consequence of Lemmas \ref{lem:Delta_t_decomp} and \ref{lem:standard-components}:

\begin{corollary}\label{cor:Delta_t_decomp_standard}
 The $S_k$-module $\Delta_k^t$ is a standardly filtered object in $\uRep^{ab}(S_t)$, and the standard components of $\Hom_{S_k}(\mu, \Delta_k^{t})$ are
$$ \mathbf{M}_{\tau}, \;\;\; \tau\subset \mu, \mu /\tau \in HS$$

\end{corollary}


\subsection{Schur-Weyl duality}\label{ssec:Deligne_SW}

An analog of classical Schur-Weyl duality for $\uRep^{ab}(S_{t})$ has been developed in \cite{En} by constructing an object $V^{\underline{\otimes} t}$ in $\Ind \mbox{-} \uRep(S_t)$ for $t \in \C$ together with an action of $\mathfrak{gl}(V)$ on it. This object is a polynomial interpolation of the usual tensor power $V^{\otimes d}$ in $\C[S_d] \otimes_{\C} \U(\mathfrak{gl}(V))$ for a vector space $V$.

We give here a short summary of the relevant results.

Let $V \simeq \C v_0 \oplus U$ be a unital vector space where $v_0$ denotes a distinguished non-zero vector in $V$. 


We then obtain the decomposition \[ \mathfrak{gl}(V) \simeq \C \id_V \oplus \mathfrak{u}_{\mathfrak{p}}^- \oplus \mathfrak{u}_{\mathfrak{p}}^+ \oplus \mathfrak{gl}(U) \] where $\mathfrak{u}_{\mathfrak{p}}^- \simeq U$ and $\mathfrak{u}_{\mathfrak{p}}^+ \simeq U^*$. Recall from lemma \ref{lem:tens_pwr_graded} that the usual tensor power $V^{\otimes d}$ is isomorphic as a $\C[S_d] \otimes U(\mathfrak{gl}(V))$-module to
\begin{equation}\label{eq:classical_tens_power}
 V^{\otimes d} \simeq \bigoplus_{k=0,\ldots,d}  \left(U^{\otimes k} \otimes \Delta_k^d \right)^{S_k}.
\end{equation}
\begin{remark}
 Consider $V$ as a graded vector space with $gr(v_0) =0$, $gr(U)=1$. Then the induced $\Z$-grading on $V^{\otimes d}$ corresponds precisely to the grading on the right hand side.
\end{remark}

 Since we have the analogs $\Delta_k^t$ of $\Delta^k_d=\C \Inj(\{1,\ldots,k\},\{1,\ldots,d\})$ in the Deligne category, we define a $\Z_{\geq 0}$-graded object 
\begin{equation}\label{eq:complex_tens_power}
 V^{\underline{\otimes} t} = \bigoplus_{k \geq 0} \left(U^{\otimes k} \otimes \Delta_k^t \right)^{S_k}.
\end{equation}
 Then $\mathfrak{gl}(V)$ acts on $V^{\underline{\otimes} t}$ as follows: $id_V$ acts by the scalar $t$, $\mathfrak{gl}(U)$ acts naturally on each summand $(U^{\otimes k} \otimes \Delta_k^t)^{S_k}$ and $\mathfrak{u}_{\mathfrak{p}}^+$, $\mathfrak{u}_{\mathfrak{p}}^-$ act by operators of degrees $1,-1$ respectively.

 
We consider the following category:

\begin{definition}
 The category $\co^{\mathfrak{p}}_{t, V}$ is defined to be the full subcategory of $\Mod_{\U(\gl(V))}$ whose objects $M$ satisfy the following conditions:
 \begin{itemize}
 \item \InnaB{$M$ is a Harish-Chandra module for the pair $(\gl(V), \Aut(V, v_0))$, i.e. the action of the Lie subalgebra $\mathfrak{u}_{\mathfrak{p}}^+ \oplus \mathfrak{gl}(U)$ on $M$ integrates to the action of the mirabolic subgroup $\Aut(V, v_0)$, with quotient $GL(U)$ acting polynomially (cf. \cite[Definition 1.2.4]{En}). 
 }
  \item $M$ is a finitely generated $\U(\gl(V))$-module.
  \item $\id_V \in \gl(V)$ acts by $t \id_M$ on $M$.
 \end{itemize}
\end{definition}

The category $\co^{\mathfrak{p}}_{t, V}$ is an Artinian abelian category, and is a Serre subcategory of the usual category $\co$ for $\gl(V)$. The $\gl(V)$-action on the object $V^{\underline{\otimes}  t}$ is a ``$\co^{\mathfrak{p}}_{t}$-type'' action, which allows us to define a contravariant functor from $\uRep^{ab}(S_t)$ to $\co^{\mathfrak{p}}_{t, V}$: \[SW_{t, V} :=  \Hom_{\uRep^{ab}(S_t)}(\cdot, V^{\underline{\otimes}  t})\]

This contravariant functor is linear and additive, yet \InnaA{only left} exact. To fix this problem, we compose this functor with the quotient functor \InnaB{$\hat{\pi}$} from $\co^{\mathfrak{p}}_{t, V}$ to the category $\widehat{\co}^{\mathfrak{p}}_{t, V}$: the localization of $\co^{\mathfrak{p}}_{t, V}$ by the Serre subcategory of finite-dimensional modules. \InnaB{We denote the newly obtained functor by $\widehat{SW}_{t, V}$.}

\begin{theorem}\cite[Theorem 1]{En}\label{SW_almost_equiv}
The contravariant functor $\widehat{SW}_{t, V}:\uRep^{ab}(S_{t}) \rightarrow {\widehat{\co}^{\mathfrak{p}}_{t, V}}$ is exact and essentially surjective.

Moreover, the induced contravariant functor $$ \quotient{\uRep^{ab}(S_{t})}{Ker(\widehat{SW}_{t, V})} \rightarrow \widehat{\co}^{\mathfrak{p}}_{t, V}$$ is an anti-equivalence of abelian categories, thus making $\widehat{\co}^{\mathfrak{p}}_{t, V}$ a Serre quotient of $\uRep^{ab}(S_{t})^{op}$.
\end{theorem}

For an explicit description of the image of the simple, standard and projective objects under $\widehat{SW}_{t, V}$ see \cite[Theorem 7.2.3]{En}. The kernel of the Serre quotient $$\uRep^{ab}(S_{t})^{op} \longrightarrow \widehat{\co}^{\mathfrak{p}}_{t, V}$$ can also be explicitly described (see \cite{En}). 


In fact, these results have been extended in \cite{En3} to the case of the unital vector space $(V = \C^{\infty},v_0 = e_1)$ to obtain an exact contravariant tensor functor $$\widehat{SW}_{t,\C^{\infty}}: \uRep^{ab}(S_{t}) \rightarrow {\widehat{\co}^{\mathfrak{p}_{\infty}}_{t, \C^{\infty}}}.$$ For the definition of ${\widehat{\co}^{\mathfrak{p}_{\infty}}_{t, \C^{\infty}}}$ see \cite[Definition 4.2.2]{En3}. In fact, the category ${\widehat{\co}^{\mathfrak{p}_{\infty}}_{t, \C^{\infty}}}$ can be obtained as a \textit{restricted inverse limit} \cite[3.4]{En3} of the inverse system of the categories ${\widehat{\co}^{\mathfrak{p}_{n}}_{t, \C^{n}}}$ with appropriate invariants functors between them.


\begin{proposition}\label{prop:SW_Deligne_infty_equivalence}
 The contravariant Schur-Weyl functor $$\widehat{SW}_{t,\C^{\infty}}: \uRep^{ab}(S_{t}) \rightarrow {\widehat{\co}^{\mathfrak{p}_{\infty}}_{t, \C^{\infty}}}$$ is an anti-equivalence of abelian categories.
\end{proposition}

\section{Objects \texorpdfstring{$\Delta_k^{\infty}$}{Delta} under the functor \texorpdfstring{$\mathbf{\Gamma}_t$}{Phi}}\label{sec:Del_correspondence}

Consider the object $\fh_t$ in $\uRep^{ab}(S_t)$. This is a Frobenius object, hence by the universal property of $\Rep(S_{\infty})$ (see Section \ref{ssec:univ_prop_S_infty}), we obtain a left exact symmetric monoidal functor \[ \mathbf{\Gamma}_t: \Rep(S_{\infty}) \to \uRep^{ab}(S_t)\] mapping $\mathfrak{h}_{\infty}$ to $\fh_t$.

\begin{remark}
It does not make sense to search for a left exact SM functor $\tilde{\mathbf{\Gamma}}_t: \Rep S_{\infty} \to \uRep(S_t)$ for $t \in \Z_{\geq 0}$, since in that case $\uRep(S_t)$ is not abelian. Theorem \ref{thrm:main} indeed implies that the essential image of $\mathbf{\Gamma}_t$ contains objects which are not in $\uRep(S_t)$.
\end{remark}

The next two Lemmas explain what the functor does to the objects $\Delta_k^{\infty}$ and to the morphisms $\res^i$ between them.

\begin{lemma}\label{lem:Delta_objects_corresp}
 The functor $\mathbf{\Gamma}_t: \Rep(S_{\infty}) \longrightarrow \uRep^{ab}(S_t)$ takes 
$\Delta_k^{\infty}$ to $\Delta_k^t$ for any $k \geq 0$, and the diagram
 $$ \xymatrix{&{S_k} \ar[r] \ar[rd] &\End_{S_{\infty}}(\Delta_k^{\infty}) 
\ar[d]^{\mathbf{\Gamma}_t}\\ &{} &\End_{\uRep(S_t)}(\Delta_k^t)}$$
 commutes.
\end{lemma}

\begin{proof}
In $\Rep(S_{\infty})$, the object $\Delta_k^{\infty}$ is a direct summand of $\mathfrak{h}_{\infty}^{\otimes k}$, given by an $S_k$-equivariant idempotent $e_k^{\infty}$. In $ \uRep^{ab}(S_t)$, $\Delta_k^t$ is a direct summand of $\mathfrak{h}_t^{\otimes k}$, given by an $S_k$-equivariant idempotent $e_k^{t}$. 

Since $\mathbf{\Gamma}_t$ is a monoidal functor, we have an $S_k$-equivariant isomorphism $$\mathbf{\Gamma}_t(\mathfrak{h}_{\infty}^{\otimes k}) \cong \mathfrak{h}_t^{\otimes k},$$ and we only need to show that $\mathbf{\Gamma}_t(e_k^{\infty}) = e_k^{t}$. This is done by writing out $e_k^{\infty}$ and $e_k^{t}$ explicitly, following \cite[Section 2]{CO} and \cite{Del07}. 

We use the basis of $\End_{S_{\infty}}(\mathfrak{h}_{\infty}^{\otimes k})$ given in Section \ref{ssec:cat_S_infty}. This basis is parametrized by equivalence relations $\pi$ on $[k] \sqcup [k] = \{1 , \ldots, k\} \sqcup \{1' , \ldots, k'\} $ such that each equivalence class meets the first copy of $[k]$. Denote the set of such equivalence relations by $\bar{P}_{k, k}^{\infty}$. By abuse of notation, for any $\pi \in \C P_{k, k}^{\infty}$ we will denote the corresponding map in $\End_{S_{\infty}}(\mathfrak{h}_{\infty}^{\otimes k})$ by $\pi$ as well. 

On the Deligne category side, we have a basis of $\End_{S_t}(\mathfrak{h}_t^{\otimes k})$ given in Section \ref{ssec:Deligne_general}. This basis is parametrized by all equivalence relations (partitions) $\pi$ on $[k] \sqcup [k]$. We denote these equivalence relations by $P_{k, k}$. Then the construction of the above bases, together with the fact that $\mathbf{\Gamma}_t$ is monoidal, implies $$\mathbf{\Gamma}_t(\pi) = \pi$$ for any $\pi \in \bar{P}_{k, k}^{\infty}$. Let us consider more generally a decomposition $$\mathfrak{h}_{\infty}^{\otimes k} \cong \bigoplus_{R} \Delta_{[k]/R}^{\infty}, \;\; \mathfrak{h}_t^{\otimes k} \cong \bigoplus_{R} \Delta_{[k]/R}^{t}$$ where $R$ runs through the set of equivalence relations on the set $[k]$ and $[k]/R$ is the set of equivalence classes (hence $\Delta_{[k]/R}^{\infty}$ is canonically isomorphic to $\Delta_{l}^{\infty} $ where $l=\abs{[k]/R}$ is the number of equivalence classes).

We will construct explicitly the corresponding idempotent in the infinite case (in the Deligne category it is done in the same way, cf. \cite[Section 2]{CO1}). In order to do so, we consider, a partial order $\geq$ on the set of equivalence relations on $[k]$: $R' \geq R$ if $R'$ is coarser (i.e. $i \sim_{R} j$ $\Rightarrow $ $i \sim_{R'} j$).

Now, for each equivalence relation $R$ on the set $[k]$, consider the direct summand of $$\mathfrak{h}_{\infty}^{\otimes k} \cong \C Fun([k], [\infty])$$ spanned by all functions $f: [k] \to [\infty]$ which are constant on the equivalence classes of $R$. This direct summand is $$\bigoplus_{R' \geq R} \Delta_{[k]/R'}^{\infty}.$$ The projection on this direct summand is an idempotent given by the minimal equivalence relation $\pi_R$ on $[k] \sqcup [k]$ which satisfies: $i \sim_{\pi_R} i'$ for every $1 \leq i \leq k$ and $i \sim_{\pi_R} j$ iff $i \sim_{R} j$. Then define recursively an idempotent $x_{R}$ in $\C P_{k, k}$ via
$$x_{R} = \pi_R - \sum_{R' \gneqq R} x_{R'}.$$ Thus, $x_{R}$ corresponds to the projection of $\mathfrak{h}_{\infty}^{\otimes k}$ onto the direct summand $\Delta_{[k]/R}^{\infty}$, and $$\mathbf{\Gamma}_t(x_{R}) = x_{R}$$ implying that $$\mathbf{\Gamma}_t \left(\Delta_{[k]/R}^{\infty}\right) \cong \Delta_{[k]/R}^{t}.$$

The required statement $$\mathbf{\Gamma}_t\left(\Delta_k^{\infty} \right) \cong \Delta_k^t $$ is then just the special case when $R = R_{sing}$ is the trivial partition of $[k]$ into singletons (i.e. minimal with respect to the partial order $\geq $), and $x_R$ corresponds to both $e_k^{\infty}$ and $e_k^t$.
\end{proof}

\begin{lemma}\label{lem:res_morph_corresp}
 The functor $\mathbf{\Gamma}_t$ takes $\res^i: \Delta_k^{\infty} \to \Delta_{k-1}^{\infty}$ to $\res^i: \Delta_k^t \to \Delta_{k-1}^t$ for any $i \in \{1, \ldots k\}$.
\end{lemma}

\begin{proof}

Consider the map $\res^i: \mathfrak{h}_{\infty}^{\otimes k} \to \mathfrak{h}_{\infty}^{\otimes k-1}$ given by $$\id^{\otimes k-1} \otimes \eta \otimes \id^{\otimes n-k}$$
where $\eta: \mathfrak{h}_{\infty} \to \C$ is the counit of the Frobenius object $\mathfrak{h}_{\infty}$. Identifying $\mathfrak{h}_{\infty}^{\otimes k} \cong \C Fun([k], [\infty])$, this corresponds to the map $$\res^i: \C Fun([k], [\infty])\longrightarrow \C Fun([k-1], [\infty])$$ given in Section \ref{ssec:prelim_rep_sym}.

Similarly, we have the map $\res^i: \mathfrak{h}_t^{\otimes k} \to \mathfrak{h}_t^{\otimes k-1}$ given by $$\id^{\otimes k-1} \otimes \eta_t \otimes \id^{\otimes n-k}$$ where $\eta_t$ is the counit of the Frobenius object $\mathfrak{h}_{\infty}$. Since $\mathbf{\Gamma}_t$ is a monoidal functor such that $\mathbf{\Gamma}_t(\eta) =\eta_t$, we have: $\mathbf{\Gamma}_t(\res^i) = \res^i$.

Using the explicit idempotents $e_k^{\infty}$ and $e_k^t$, one immediately sees that the restriction of $\res^i$ to the direct summand $ \Delta_k^{\infty} \subset \mathfrak{h}_{\infty}^{\otimes k}$ gives precisely the map $$\res^i: \Delta_k^{\infty} \to \Delta_{k-1}^{\infty}$$ and the restriction of $\res^i$ to $\Delta_k^t$ gives precisely $\res^i: \Delta_k^t \to \Delta_{k-1}^t$. Hence $\mathbf{\Gamma}_t(\res^i): \mathbf{\Gamma}_t(\Delta_k^{\infty}) \to \mathbf{\Gamma}_t(\Delta_{k-1}^{\infty})$ equals $\res^i: \Delta_k^t \to \Delta_{k-1}^t$, as required.

\end{proof}


\section{Delta complex in \texorpdfstring{$\Rep(S_\infty)$}{Rep(S)}}\label{sec:Delta_complex_infty}

\subsection{Category of \texorpdfstring{$\mathbf{FI}$}{FI}-modules and the objects \texorpdfstring{$\Delta_k^{\infty}$}{Delta}}\label{ssec:Delta_complex_FI}

We will use the notation and results from Section \ref{sec:infinity}. For $V = \C^{\infty}$ let $A = \Sym V$, and consider the category $\Mod_A$ of finitely-generated
$A$-modules in $\Ind \mbox{-} \V$. We will use the notation $$\widetilde{\otimes} := \otimes_A$$ to denote the tensor product in this category.

\begin{remark}
 One should keep in mind that unlike the tensor product $\otimes$ in $\V$ which is exact, the tensor product $\widetilde{\otimes} = \otimes_A$ is only left-exact.
\end{remark}

Then the functor $${\Phi}: \V \xrightarrow{\sim} \bigoplus_N \Rep(S_N)$$ induces an equivalence (see Theorem \ref{thm:fi-mod}) $$\widehat{\Phi}: \Mod_A \xrightarrow{\sim} {\bf FI} \mbox{-} \Mod$$ of abelian categories. Once again, one should bear in mind that this is not an equivalence of monoidal categories.

Let us now consider, for each $n \geq 0$, the objects $(A \otimes V)^{\widetilde{\otimes} n} \cong A \otimes 
V^{\otimes n}$, together with the action of the group $S_n$ on them. We will study their images under the 
equivalence $\widehat{\Phi}$.

\begin{lemma}\label{lem:Phi_action_tensor_powers}
The functor $\widehat{\Phi}$ takes the object $A \otimes 
V^{\otimes n}$ to the {\bf FI}-module $$\bigoplus_{k \geq 0} 
\Ind^{S_{k+n}}_{S_k \times S_n} \left( \mathbb{C} \otimes \mathbb{C}[S_n] \right) \cong 
\bigoplus_{N \geq n} \mathbb{C} \Inj \left([n],[N]\right)= \bigoplus_{N \geq n} \Delta_n^N.$$

The action of $S_n$ on $(A \otimes V)^{\widetilde{\otimes} n}$ then corresponds to the 
action of $S_n$ on each of the $\Delta_n^N$.
\end{lemma}
\begin{proof}
We will say that an object $U$ in $\V$ is homogeneous of degree $d$ if $U$ belongs to the full Karoubian additive subcategory generated by $V^{\otimes d}$, i.e. $U$ is a direct sum of simple modules of the form $S^{\lambda} V$, $\lam \vdash d$.

Recall from \cite{SS-2} that for any two homogeneous objects $U_1, U_2  \in \V$ of degrees $d_1, d_2$, we have $$\Phi(U_1 \otimes U_2) = \Ind_{S_{d_1} \times S_{d_2}}^{S_{d_1 + d_2}} \Phi(U_1) \otimes \Phi(U_2).$$

Using the fact that $\Phi(S^{\lambda} V) = \lambda$ for any $\lambda$, we then obtain $${\Phi}(A) \cong \bigoplus_{k \geq 0}  \mathbb{C} , \;\; {\Phi}(V^{\otimes n}) \cong  \mathbb{C}[S_n]$$ where $\mathbb{C}[S_n]$ is considered an {\bf FI}-module sitting in degree $n$ (i.e. the left regular $S_n$-representation), and the isomorphism ${\Phi}(V^{\otimes n}) \cong  \mathbb{C}[S_n]$ is $S_n$-equivariant with respect to the right regular $S_n$-action on $\C[S_n]$.  

This proves the existence of an $S_n$-equivariant isomorphism $$\widehat{\Phi}(A \otimes 
V^{\otimes n}) \cong \bigoplus_{k \geq 0} \Ind^{S_{k+n}}_{S_k \times S_n} \left( \mathbb{C} \otimes \mathbb{C}[S_n] \right)$$ as required.
\end{proof}

\begin{lemma}\label{lem:Phi_action_mu_maps}
 Consider the multiplication map $$\mu: A\otimes V \cong \Sym(V)\otimes V \to \Sym(V) =A.$$ Then $\widehat{\Phi}(\mu)$ is the 
morphism $$\bigoplus_{N \geq 1} 
\Delta_1^N = \bigoplus_{N \geq 1} \mathbb{C}^N \longrightarrow \bigoplus_{N 
\geq 0} \Delta_0^N \cong \bigoplus_{N \geq 0} \mathbb{C}$$ given by the $S_N$-maps 
$\eta_N: \mathbb{C}^N \to \mathbb{C}$, $e_i \to 1$ ($i =1,\ldots, N$). 
\end{lemma}
\begin{proof}
\InnaA{Consider the $N$-th component ${\Phi}(\mu)_N$ of the map $\Phi(\mu)$.} The map $${\Phi}(\mu)_N: \Delta_1^N = \mathbb{C}^N \longrightarrow  \mathbb{C}$$ is by definition $S_N$-equivariant. Since $\mu$ is surjective, and $\Phi$ is an equivalence, we conclude that $\Phi(\mu)$ is surjective as well, and hence $\widehat{\Phi}(\mu)_N \cong \eta_N$ (the latter being the unique non-zero $S_N$-equivariant map $\mathbb{C}^N \to \mathbb{C}$).
\end{proof}
More generally, consider $$\mu^{(i)} = \id^{\otimes i-1} \otimes \mu \otimes 
\id^{\otimes n-i}:  (A \otimes V)^{\widetilde{\otimes} n} \longrightarrow (A \otimes 
V)^{\widetilde{\otimes} n-1}.$$

\begin{corollary}\label{cor:Phi_action_mu_maps}
The map $\widehat{\Phi}(\mu^{(i)})$ is the morphism $$\res^i: \bigoplus_{N \geq n} 
\Delta_n^N \longrightarrow \bigoplus_{N \geq n-1} \Delta_{n-1}^N  $$ given by 
maps $\res^i: \C \Inj([n], [N]) \to \C \Inj([n-1], [N])$. 
\end{corollary}

\begin{proof}
Using the fact that $$\widehat{\Phi}(A \otimes 
V^{\otimes n}) \cong \bigoplus_{N \geq n} \Delta_n^N$$ is an $S_n$-equivariant isomorphism, it is enough to prove the claim in the case $i=1$. The claim then follows from Lemma \ref{lem:Phi_action_mu_maps}.
\end{proof}

Now, the quotient functor $$q:\mathbf{FI} \mbox{-} \Mod \longrightarrow  \Rep(S_{\infty})$$ 
takes $\bigoplus_{N \geq n} \Delta_n^N $ to $\Delta_n^{\infty}$ and the maps 
$\res^i: \bigoplus_{N \geq n} \Delta_n^N \longrightarrow \bigoplus_{N \geq n-1} 
\Delta_{n-1}^N  $ to $\res^i: \Delta_n^{\infty} \longrightarrow 
\Delta_{n-1}^{\infty}$. 

\begin{corollary}\label{cor:funct_A_mod_S_infty}
 Consider the composition of functors $$ \Mod_A \xrightarrow{\widehat{\Phi}} \mathbf{FI} \mbox{-} \Mod \xrightarrow{q}  \Rep(S_{\infty}).$$ Then for any $n$ and $i$, 
$$q\circ \widehat{\Phi} \left( (A \otimes V)^{\widetilde{\otimes} n}\right)  \cong \Delta_n^{\infty}, \;\;\; q\circ \widehat{\Phi} \left(\mu^{(i)} \right)=\res^i.$$
\end{corollary}

\begin{definition}[Torsion $A$-modules]\label{def:torsion_module}
 An $A$-module $M$ is called a {\it torsion} module (sometimes also called ``finite-length'') if it has finite length as an object in the semisimple category $\Ind \mbox{-} \V$ (in other words, it is a proper object in $\V$).
\end{definition}

Consider now the Serre quotient of $\Mod_A$ by the subcategory of torsion $A$-modules: $$\Mod_A \longrightarrow  \quotient{\Mod_A}{\Mod_A^{\tors}}.$$ 

As it was shown in \cite[Theorem 6.2.4]{SS}\cite[Proposition 1.3.5]{SS-2}, the composition $q\circ \widehat{\Phi}$ factors through $\quotient{\Mod_A}{\Mod_A^{\tors}}$ and induces an equivalence 
$$ \widehat{\Phi}: \quotient{\Mod_A}{\Mod_A^{\tors}} \xrightarrow{\sim} \Rep(S_{\infty}).$$

\subsection{The \texorpdfstring{complex $K_{n, \infty}^{\bullet}$ in $\Rep(S_\infty)$}{Delta complex in Rep(S)}}\label{ssec:Delta_complex_infty}
Consider in $\Mod_A$ the complex $$C^{\bullet} = 0 \to A\otimes V \longrightarrow A \to 0$$ 
where the differential is $\mu$, and $A\otimes V$ sits in degree $-1$. This complex has homology $$H^{-1} 
C^{\bullet}= \Im(A \otimes \Lambda^2 V \to A \otimes V), \;\; H^0 C^{\bullet}= 
\mathbb C$$

Now consider the complex $(C^{\bullet})^{\widetilde{\otimes} n}$. This complex has the form 
$$0 \to A \otimes V^{\otimes  n} \to \binom{n}{n-1} A \otimes V^{\otimes n -1}  
\to \binom{n}{n-2} A \otimes V^{\otimes n-2} \to \ldots \to \binom{n}{1} A 
\otimes V  \to  A \to 0$$

We now endow the complex $(C^{\bullet})^{\widetilde{\otimes} n}$ with a natural $S_n$ action.
Indeed, if we consider the action of $S_n$ on $(C^{\bullet})^{\widetilde{\otimes} 
n}$ via permutation of tensor factors, then $(C^{\bullet})^{\widetilde{\otimes} 
n}$ is isomorphic to the following complex (see Lemma \ref{lem:tens_pwr_cmplx}): 
$$0 \to  \left(\bdel_n^n \boxtimes A \otimes V^{\otimes n} \right)^{S_n} \to 
\left(\bdel_{n-1}^n \boxtimes  A \otimes V^{\otimes n -1}\right)^{S_{n-1}}  
\to \ldots \to  A \otimes V  \to   A \to 0$$
That is, the $(-k)$-th component of $(C^{\bullet})^{\widetilde{\otimes} n}$ is in fact the 
$\mathbb{C}[S_n] \boxtimes A$-module $\left( \bdel_k^n \boxtimes  A \otimes 
V^{\otimes k} \right)^{S_{k}}$. The differentials in this complex  
$$\partial: \left( \bdel_k^n \boxtimes  A \otimes V^{\otimes k} 
\right)^{S_{k}} \longrightarrow \left( \bdel_{k-1}^n \boxtimes  A \otimes 
V^{\otimes k-1} \right)^{S_{k-1}}$$ are given by $\partial = \sum_{i=1}^k (-1)^i 
\res^i \otimes \mu^{(i)}$.

Consider the image of the complex 
$(C^{\bullet})^{\widetilde{\otimes} n}$ with the $S_n$ action under the quotient functor $q\circ \widehat{\Phi}$. The obtained complex in $\Rep(S_{\infty})$ is denoted by $K_{n, \infty}^{\bullet}$, and will be called {\it the Delta complex}:
$$K_{n, \infty}^{\bullet}: 0 \to \left(\bdel_n^n \boxtimes \Delta_n^{\infty} \right)^{S_n} \to 
\left(\bdel_{n-1}^n \boxtimes  \Delta_{n-1}^{\infty}\right)^{S_{n-1}}  \to  \ldots 
\to  \Delta_1^{\infty}  \to   \Delta_0^{\infty} \to 0 $$
The $(-k)$-th component of this complex is the $S_n \times 
S_{\infty}$-module $\left( \bdel_k^n \boxtimes  \Delta_{k}^{\infty} 
\right)^{S_{k}}$. The differentials in this complex
$$\partial: \left( \bdel_k^n \boxtimes   \Delta_{k}^{\infty} \right)^{S_{k}} 
\longrightarrow \left( \bdel_{k-1}^n \boxtimes   \Delta_{k-1}^{\infty} 
\right)^{S_{k-1}}$$ are given by $\partial = \sum_{i=1}^k (-1)^i \res^i \otimes 
\res^i$.

\subsection{Cohomology of the Delta complex}\label{ssec:cohomology_Delta_infty}
Let $V = \C^{\infty}$ with basis $v_0, v_1, \ldots$, and set $U = span_{\C}\{ v_1, v_2, \ldots\}$. Denote by $\epsilon: V \to\C $ the map $ v_0 \mapsto 1 $.

We consider the space
$$V^{\underline{\otimes} \infty} := \bigoplus_{m \geq 0} (\Delta_m^{\infty} \otimes U^{\otimes m})^{S_m} \cong \varinjlim V^{\otimes n} $$
by analogue with \eqref{eq:classical_tens_power} and the construction of the complex tensor power in \cite{En}; see also \cite[Section 6.2]{SS}.

This is a $GL(U) \times S_{\infty}$-module, which has an additional action by the (additive) group $U^*$. Notice that by definition it is an injective object in $Ind-\Rep(S_{\infty})$.

We consider the (contravariant) functor $$SW_V: \Rep(S_{\infty})^{op} \to Ind-Rep(GL(U)), \;\;\; M \to \Hom_{S_{\infty}}(M, V^{\underline{\otimes} \infty})$$

\begin{lemma}
 The functor $SW_V$ is exact and faithful, and the image lies in $Rep(GL(U))$.
\end{lemma}

\begin{remark}
 In fact, one can show that $SW_V$ defines an equivalence between $\Rep(S_{\infty})$ and the category of $GL(U)$-polynomial $GL(U) \ltimes U^*$-representations.
\end{remark}

\begin{proof}
 Exactness follows from the fact that $V^{\underline{\otimes} \infty}$ is injective object in $Ind-\Rep(S_{\infty})$. Since it is exact, to show that this functor is faithful it is enough to check that $SW_V(L_{\tau}) \neq 0$ for any simple object $L_{\tau} \in \Rep(S_{\infty})$. Indeed, by Lemma \ref{lem:Delta_decomposition_infty}, we have an isomorphism of $S_k$-modules $$\Hom_{S_{\infty}} (L_{\tau}, \Del^{\infty}_k) = \bigoplus_{\lambda \vdash k} \lambda \otimes \Hom_{S_{\infty}}(L_{\tau}, Q_{\lambda})$$
 The latter is one-dimensional iff $\lambda = \tau$, hence 
 $$\Hom_{S_{\infty}} (L_{\tau}, \Del^{\infty}_k) = \begin{cases}
                                                    \tau &\text{  if  } \tau \vdash k \\
                                                    0 &\text{ else}
                                                   \end{cases}
$$

and 
$$SW_V(L_{\tau}) \cong \Hom_{S_{\infty}}(L_{\tau}, V^{\underline{\otimes} \infty}) \cong (\tau \otimes U^{\otimes \abs{\tau}})^{S_{\abs{\tau}}} \cong S^{\tau} U \neq 0 $$ as required.
 
\end{proof}

Let us check what happens to the complex $K_{n, \infty}^{\bullet}$ under the functor $SW_V$. 
\begin{proposition}
 The complex $SW_V(K_{n, \infty}^{\bullet})$ has no cohomology except in top degree.
\end{proposition}
\begin{proof}

We first compute $SW_V(\Delta_k^{\infty})$. Recall from Section \ref{ssec:Delta_infty} that we have an isomorphism of $S_k \times S_m$-modules $$\Hom_{S_{\infty}}\left( \Delta^{\infty}_k, \Delta^{\infty}_m \right) \cong \Delta^k_m$$ This induces an isomorphism of $S_k$-modules

\begin{align*}
& SW_V(\Delta_k^{\infty}) \cong \bigoplus_{m \geq 0} \left(\Hom_{S_{\infty}}\left( \Delta^{\infty}_k, \Delta^{\infty}_m \right) \otimes  U^{\otimes m}\right)^{S_m} \cong  \bigoplus_{m \geq 0}  \left(\Delta^k_m  \otimes  U^{\otimes m}\right)^{S_m} \cong V^{\otimes k}
\end{align*}

The map $SW_V(\res^i: \Delta_k^{\infty} \to \Delta_{k-1}^{\infty})$ is then given by $\eps^{(i)}: V^{\otimes k-1} \to V^{\otimes k}$, where $$\eps^{(i)} = \id \otimes \ldots \otimes \id \otimes \eps \otimes \id \otimes \ldots  \otimes \id$$ (insertion of $v_0$ in $i$-th position).

Hence the complex $SW_V \left( K_{n, \infty}^{\bullet} \right)$ of $S_n \times GL(U)$-modules has $k$-th component 
$\left( \bdel_k^n \boxtimes  V^{\otimes k} \right)^{S_{k}}$. The differentials in this complex  
$$\partial: \left( \bdel_k^n \boxtimes  V^{\otimes k}\right)^{S_{k}} \longrightarrow \left( \bdel_{k+1}^n \boxtimes 
V^{\otimes k+1} \right)^{S_{k+1}}$$ are given by $\partial = \sum_{i=1}^k (-1)^i 
(\res^i)^* \otimes \eps^{(i)}$. But this is just the complex $(\C \xrightarrow{\eps} V)^{\otimes n} $ with the $S_n$-action given by permuting the factors.

Thus we have an isomorphism of $S_n \times GL(U)$-complexes
$$SW_V \left( K_{n, \infty}^{\bullet} \right) \simeq (\C \xrightarrow{\eps} V)^{\otimes n} $$

Now, recall that $\otimes$ is bi-exact, and the complex $\C \xrightarrow{\eps} V$ has cohomology only in top degree. Hence its tensor powers have cohomology only in the top $n$-th degree.

\end{proof}

Using the fact that $SW_V$ is exact and faithful, we conclude:
\begin{corollary}\label{cor:lower_cohom_Delta_complex_infty}
 The complex $K_{n, \infty}^{\bullet}$ has no cohomology in degrees larger than $-n$.
\end{corollary}

\begin{corollary}
 The cohomology of the complex $(C^{\bullet})^{\widetilde{\otimes} n}$ in degrees larger than $-n$ consists of torsion $A$-modules.
\end{corollary}
\subsection{Top cohomology of the Delta complex}\label{ssec:top_cohomology_Delta_infty}
The following lemma will be needed to compute the top cohomology of the complex:
\begin{lemma}\label{lem:decomp_prod_delta_infty_case}
 We have an isomorphism of $S_{\infty} \times S_n$-modules $$\left( \bdel_k^n \boxtimes  \Delta_{k}^{\infty} \right)^{S_{k}} = \bigoplus_{\substack{\mu \vdash n, \; \lambda \vdash k, \\ \lambda \subset \mu^{\vee}, \; \mu^{\vee} / \lambda \in VS} } Q(\lambda) \otimes \mu.$$
In particular for any $\mu \vdash n$, the simple components (subquotients) of $$\Hom_{S_n}(\mu, \left( \bdel_k^n \boxtimes  \Delta_{k}^{\infty} \right)^{S_{k}}) $$ in the category $\Rep(S_{\infty})$ are of the form $L_{\tau}$, where $ \abs{\tau}\leq k$ and $ \tau \subset \mu^{\vee}$. 

The multiplicity of $L_{\tau}$ in $\Hom_{S_n}(\mu, \left( \bdel_k^n \boxtimes  \Delta_{k}^{\infty} \right)^{S_{k}}) $ is $$\sharp \{\lambda \vdash k: \tau\subset \lambda\subset\mu^{\vee}, \, \mu^{\vee} / \lambda \in VS, \, \lambda / \tau \in HS \}.$$ 
\end{lemma}
\begin{proof}
 As we have seen in Lemma \ref{lem:Delta_decomposition_infty}, there is an isomorphism of $S_\infty \times S_k$-modules $$ \Delta_{k}^{\infty} \cong \bigoplus_{\lambda \vdash k} Q(\lambda) \otimes \lambda. $$ By Lemma \ref{lem:Delta_classical_decomp}, there is also an isomorphism of $S_n \times S_k$-modules 
$$ \bdel_{k}^{n} \cong \bigoplus_{\substack{\mu \vdash n, \; \lambda \vdash k, \\ \lambda \subset \mu^{\vee}, \; \mu^{\vee} / \lambda \in VS} } \mu \otimes \lambda$$
This implies the decomposition $$\left( \bdel_k^n \boxtimes  \Delta_{k}^{\infty} \right)^{S_{k}}  = \bigoplus_{\substack{\mu \vdash n, \; \lambda \vdash k, \\ \lambda \subset \mu^{\vee}, \; \mu^{\vee} / \lambda \in VS} } Q(\lambda) \otimes \mu.$$ The desired property of the simple components now follows from Proposition \ref{prop:infty_injective_obj}.
\end{proof}

\begin{proposition}\label{prop:Delta_cx_infty_hmgy}
For $k <n$, $H^{-k}K_{n, \infty}^{\bullet} =0$, and $$H^{-n}(K_{n, \infty}^{\bullet}) = \bigoplus_{\mu \vdash n} \mu^{\vee} \otimes L_{\mu} $$ as an $ S_n \times S_{\infty}$-module. 
\end{proposition}

\begin{proof}

Recall that by Corollary \ref{cor:lower_cohom_Delta_complex_infty}, for $k <n$ we have: $H^{-k}K_{n, \infty}^{\bullet} =0$.

In order to compute the $(-n)$-th cohomology, we analyze the simple components which appear in it by comparing the $(-n)$-th term of the complex $K_{n, \infty}^{\bullet}$ to the Euler characteristic of this complex (taking into account the $S_n$-action on this characteristic).

Let $\mu \vdash n$ and consider the simple components of $ \Hom_{S_n}(\mu, K_{n, \infty}^{-k})$.


By Proposition \ref{lem:decomp_prod_delta_infty_case} the multiplicity of $L_{\tau}$ is  $$\sharp \{\lambda \vdash k: \tau\subset \lambda\subset\mu^{\vee}, \, \mu^{\vee} / \lambda \in VS, \, \lambda / \tau \in HS \}.$$


In particular, for $k=n$, we see that the simple components of $ \Hom_{S_n}(\mu, K_{n, \infty}^{n})$ are $ L_{\tau}$ where $\tau$ satisfies the condition
\begin{equation}\label{eq:cond_top_homology}
 \mu \vdash n, \; \tau\subset \mu^{\vee}, \; \mu^{\vee} / \tau \in HS 
\end{equation}
Hence the $(-n)$-th cohomology can only have simple components of this form. 

\InnaA{We denote the set of partitions $\tau$ such that $\tau$ satisfies \eqref{eq:cond_top_homology} by $T$.}


Let us consider the Euler characteristic $\chi$ of the complex $ \Hom_{S_n}(\mu, K_{n, \infty}^{\bullet})$ in the Grothendieck group of $\Rep(S_{\infty})$. 
\begin{equation}\label{eq:euler_char_infty}
  \chi(\Hom_{S_n}(\mu, K_{n, \infty}^{\bullet}))= \sum_{\tau \in T} \sum_{k \geq 0} (-1)^{-k} \sharp \{\lambda \vdash k: \tau\subset \lambda\subset \mu^{\vee}, \, \mu^{\vee} / \lambda \in VS, \, \lambda / \tau \in HS \}[ L_{\tau} ].
\end{equation}


Fix \InnaA{$\tau \in T$}. Let $j_1, \ldots, j_s$ be the indexes of the non-empty rows in the horizontal strip $\mu^{\vee} / \tau$. 

Then the Young diagrams $\lambda$ satisfying the conditions in \eqref{eq:euler_char_infty} are those \InnaA{where the vertical strip $\mu^{\vee} / \lambda$ is a subset of the horizontal strip $\mu^{\vee} / \tau$. That is, $\lambda \vdash k$ is obtained from $\mu^{\vee}$ by removing at most one box from} each of the rows $j_1, \ldots, j_s$, and altogether removing $n-k$ boxes. Denoting $l=n-k$, we obtain: 
$$\sum_{k \geq 0} (-1)^{(-k)} \sharp \{\lambda \vdash k \mid \substack{\tau\subset \lambda\subset \mu^{\vee}, \\ \mu^{\vee} / \lambda \in VS, \; \lambda / \tau \in HS }\} = \sum_{l \geq 0}^{s} (-1)^{l-n}  \binom{s}{l} = (-1)^n\delta_{s, 0}$$

Hence for $\tau$ satisfying conditions \eqref{eq:cond_top_homology}, $[L_{\tau} ]$ would appear with multiplicity zero in the Euler characteristic $\chi(\Hom_{S_n}(\mu, K_{n, \infty}^{\bullet}))$ iff $\tau \subsetneq \mu^{\vee}$, and with multiplicity $1$ iff $\tau = \mu^{\vee}$.

We conclude that $$H^{-n}(K_{n, \infty}^{\bullet}) = \bigoplus_{\mu \vdash n} \mu^{\vee}\otimes L_{\mu},$$ and the proposition is proved.
\end{proof}

\begin{corollary}\label{cor:Delta_cx_infty_resol}
 The complex $\Hom_{S_n}( \lambda^{\vee}, K_{n, \infty}^{\bullet})$ is an injective resolution of the simple $S_{\infty}$-module $L_{\lambda}$.
\end{corollary}
We will use the following notation: $$I^{\bullet}_{\lambda}=\Hom_{S_n}( \lambda^{\vee}, K_{n, \infty}^{\bullet}).$$
In fact, this is precisely the injective resolution for $L_{\lambda}$ described in Proposition \ref{prop:infty_injective_obj}.

\section{The \texorpdfstring{complex $K_{n, N}^{\bullet}$}{Delta complex} in the classical setting}\label{sec:Delta_complex_classical}
Let us consider what happens in the classical Schur-Weyl duality with the analogous complex $K_{n, N}^{\bullet}$ in $\Rep(S_N)$. This complex is defined as
$$ K_{n, N}^{\bullet} : 0 \to  \left(\bdel_n^n \boxtimes \Delta_n^{N} \right)^{S_n} \to 
\left(\bdel_{n-1}^n \boxtimes  \Delta_{n-1}^{N}\right)^{S_{n-1}}  \to  \ldots 
\to  \Delta_1^{N}  \to   \Delta_0^{N} \cong \mathbb C \to 0 $$
The $(-k)$-th component of this complex is the $S_n \times 
S_{N}$-module $\left( \bdel_k^n \boxtimes  \Delta_{k}^{N} 
\right)^{S_{k}}$. The differentials in this complex
$$\partial: \left( \bdel_k^n \boxtimes   \Delta_{k}^{N} \right)^{S_{k}} 
\longrightarrow \left( \bdel_{k-1}^n \boxtimes   \Delta_{k-1}^{N} 
\right)^{S_{k-1}}$$ are given by $\partial = \sum_{i=1}^k (-1)^i \res^i \otimes 
\res^i$.

\subsection{Homology of the \texorpdfstring{complex $K_{n, N}^{\bullet}$}{Delta complex}}\label{ssec:derived}
Let $\mu \vdash n$, and consider the complex $$K_{\mu, N}^{\bullet} := \Hom_{S_n}(\mu^{\vee}, K_{n, N}^{\bullet}).$$ 

Then $$K_{\mu, N}^{\bullet} = \Gamma_N \Hom_{S_n}( \mu^{\vee}, K_{n, \infty}^{\bullet}) =\Gamma_N \left( I^{\bullet}_{\mu} \right)$$ where $\Gamma_N: \Rep(S_{\infty}) \to \Rep(S_N)$ is the specialization functor described in Section \ref{ssec:univ_prop_S_infty}, while $I^{\bullet}_{\mu}$ is the injective resolution of $L_{\mu}$ in $\Rep(S_{\infty})$ (see Corollary \ref{cor:Delta_cx_infty_resol}).

Then for any $N, k \geq 0$, we have: $$H^{-k} K_{\mu, N}^{\bullet} = R^{n-k} \Gamma_N L_{\mu}$$ where $R^{n-k} \Gamma_N$ is the right derived functor of the specialization functor $\Gamma_N$. The effect of $R^{n-k} \Gamma_N$ on simple modules is known and described in \cite[(6.4.6)]{SS}. Let us give a short description here as well.

Recall that a hook in a diagram $\mu$ whose vertex is $(i, j)$ is the set of all boxes $(i', j')$ of $\mu$ such that $i \geq i', j=j'$ or $i = i', j \geq j'$. The number of squares in the hook for which $i \geq i', j=j'$ is the height (leg) of the hook. For a hook $h$ in $\mu$, we denote by $\mu /h$ the Young diagram obtained by removing the hook $h$ from $\mu$.

\begin{remark}
 Removing from $\mu$ a hook whose vertex is $(i, j)$ is the same as removing a connected border strip starting with the last box in row $i$ and ending with the last box in column $j$.
\end{remark}

\begin{example}
Here are examples of two hooks, whose cells are denoted by $\times$. The left one has vertex $(1, 3)$ and the right one has vertex $(3,3)$:
$$\young(\hfil\hfil\times\times\times,\hfil\hfil\times\hfil,\hfil\hfil\times\hfil,\hfil\hfil,\hfil,\hfil) \;\;\;\;\; \;\;\;\;\; \young(\hfil\hfil\hfil\hfil\hfil,\hfil\hfil\hfil\hfil,\hfil\hfil\times\times,\hfil\hfil,\hfil,\hfil)$$
These hooks have heights $3$ and $1$ respectively, and removing these from the original diagram gives the Young diagrams 
$$\young(\hfil\hfil\hfil,\hfil\hfil\hfil,\hfil\hfil,\hfil\hfil,\hfil,\hfil) \;\;\;\;\; \;\;\;\;\; \young(\hfil\hfil\hfil\hfil\hfil,\hfil\hfil\hfil\hfil,\hfil\hfil,\hfil\hfil,\hfil,\hfil)$$
\end{example}

The $S_N$-module $R^{m} \Gamma_N (L_{\mu})$ is non-zero if and only if there exists a hook $h$ in $\mu$ satisfying the following properties: 

\begin{enumerate}
 \item $\mu / h$ is a Young diagram of size $N$
\item $h$ has height $m$, 
\item the vertex of $h$ lies in row $1$. 
\end{enumerate}
For each $\mu$, this may occur for at most one value of $m >0$ (in particular, this does not occur when $ N \geq n$). This is precisely when $N$ belongs to the strictly increasing sequence $$\{ \abs{\mu} - \mu_1 - \mu^{\vee}_j+j \}_{j \geq 1}$$ and $m = \mu^{\vee}_j$ for the appropriate $j$.

For any $\mu$, denote by $V_\mu $ the following representation: 
\begin{itemize}
 \item If there exists $j$ such that $\abs{\mu} - \mu_1 - \mu^{\vee}_j+j  =N$, let $h$ be the hook with vertex $(1, j)$, let $\lambda = \mu /h$, and set $V_{\mu} = \lambda$ be the corresponding representation of $S_N$. We will denote by $m= \mu^{\vee}_j$ the height of $h$ in this case.
\item Otherwise, set $V_\mu =0$, and $m = \infty$.
\end{itemize}

\begin{corollary}\label{cor:Delta_cx_class_hmgy}
%
For any $n$, $$H^{-n} K_{n, N}^{\bullet} =\bigoplus_{\mu \vdash n} \mu[N] \otimes \mu^{\vee}$$ where $\mu[N]$ is the Young diagram obtained from $\mu$ by adding a top row of length $N-\abs{\mu}$.

Furthermore, for $$H^{-k} K_{\mu, N}^{\bullet} \cong R^{n-k} \Gamma_N (L_{\mu}) =\begin{cases}
                                                                                 0 &\text{ if } k \neq n-m, \\
										  V_{\mu} &\text{ if } k = n-m.
                                                                                \end{cases}$$

In particular, for $ N \geq n$, we have: $H^{-k} K_{n, N}^{\bullet} =0$ for $k<n$. 
\end{corollary}
This answers the second question posed by Deligne in \cite[(7.13)]{Del07}.

\begin{remark}
 The representation $\bigoplus_{\mu \vdash n} \mu[N] \otimes \mu^{\vee}$ is the maximal direct summand of $\bdel^N_n$ which has no non-trivial maps into $\bdel^N_k$ for any $k<n$ (in the notation of \cite[Section 7]{Del07}, this is precisely $[n]^* \otimes \varepsilon$).
\end{remark}

\subsection{Schur-Weyl duality and the \texorpdfstring{complex $K_{n, N}^{\bullet}$}{Delta complex}}
Consider the (covariant) Schur-Weyl functor $$SW_N = ( - \otimes W^{\otimes N})^{S_N}: \Rep(S_N) \to \Rep(GL(W))_{pol}$$ for some finite-dimensional vector space $W$. We now compute where this functor takes $\Delta_k^N$. 

\begin{lemma}\label{lem:SW_funct_class_Delta}
We have an isomorphism of $S_k \times GL(W)$-modules
 $$SW_N(\Delta_k^N) = (\Delta_k^N \otimes W^{\otimes N})^{S_N} \cong \Sym^{N-k} W \otimes W^{\otimes k}$$
where $S_k$ acts on RHS by permuting the tensor factors.
\end{lemma}
\begin{proof}
As before, we identify $\Delta^k_N = \mathbb{C} \Inj([k], [N])$. For $f:\{1,\ldots, k\} \hookrightarrow \{1,\ldots, N\}$ and $w_1, \ldots, w_N \in W$ 
we define the following map
\begin{align*}
 \Delta_k^N \otimes W^{\otimes N} &\longrightarrow \Sym^{N-k} W \otimes W^{\otimes k} \\
f \otimes w_{1} \otimes \ldots w_{N} &\longmapsto \left(\prod_{j \notin \Im(f)} w_{j} \right) \otimes w_{{f(1)}} \otimes w_{{f(2)}} \otimes \ldots \otimes w_{{f(k)}}
\end{align*}

One immediately sees that this map is surjective and factors through the invariants $(\Delta_k^N \otimes W^{\otimes N})^{S_N}$. Comparing $S_k \times GL(W)$-characters, we obtain the required isomorphism.
\end{proof}
Denote by $\mu: \Sym(W) \otimes W \to \Sym(W)$ the multiplication map, and for $i \in \{1, \ldots, k\}$, denote by $\mu^{(i)}: \Sym(W) \otimes W^{\otimes k} \to 
\Sym(W)$ the map $\mu$ applied to the $i$-th copy of $W$. We will also denote by $gr_N{\mu}^{(i)}$ the restriction of $\mu^{(i)}$ to the $N$-th graded piece of the graded $GL(W)$-module $\Sym W \otimes W^{\otimes k}$ (here $W$ has degree $1$).
\begin{lemma}\label{lem:SW_funct_class_res}
 We have: $SW_N(\res^i: \Delta_k^N \to \Delta_{k-1}^N) = gr_N{\mu}^{(i)}$.
\end{lemma}
\begin{proof}
Consider the map $\Delta_{k-1}^N \otimes W^{\otimes N} \longrightarrow \Sym^{N-k+1} W \otimes W^{\otimes k-1}$ given in the proof of Lemma \ref{lem:SW_funct_class_Delta}. We compose it with the map $\res^i: \Delta_k^N \otimes W^{\otimes N} \to \Delta_{k-1}^N \otimes W^{\otimes N}$ and obtain the induced map $SW_N(\res^i: \Delta_k^N \to \Delta_{k-1}^N)$ given by
\begin{align*}
 \Sym^{N-k} W \otimes W^{\otimes k} &\longrightarrow \Sym^{N-k+1} W \otimes W^{\otimes k-1}, \\
 (w_1^{\alpha_1} w_2^{\alpha_2} \ldots )\otimes w'_1 \otimes \ldots w'_i \ldots \otimes w'_k &\longmapsto  (w_1^{\alpha_1} w_2^{\alpha_2} \ldots  \cdot w'_i )\otimes w'_{1} \otimes \ldots \widehat{w'_i} \ldots \otimes w'_{k}
\end{align*}
This is precisely $ gr_N{\mu}^{(i)}$, which proves the lemma.
\end{proof}

Thus the complex $SW_N(K_{n, N}^{\bullet})$ is isomorphic to the complex $$ \Sym^{N-n} W \otimes \left(W^{\otimes n} \otimes \bdel_n^n \right)^{S_n} \rightarrow   \ldots \to \Sym^{N-1}W \otimes W \to \Sym^N W$$ with degree $(-k)$ given by $\Sym^{N-k} W \otimes \left(W^{\otimes k} \otimes \bdel_k^n \right)^{S_k}$, and the differential $$\partial:\Sym^{N-k} W \otimes \left(W^{\otimes k} \otimes \bdel_k^n \right)^{S_k} \longrightarrow \Sym^{N-k+1} W \otimes \left( W^{\otimes k-1}  \otimes \bdel_{k-1}^n \right)^{S_{k-1}}$$ given by $\sum_i (-1)^i gr_N{\mu}^{(i)} \otimes \res^i$.

We now interpret this complex as part of a larger graded complex, similar to the one we have seen already. Consider the complex $C_W^{\bullet}$ given by $ \Sym(W) \otimes W \to \Sym(W)$ in degrees $0,1$. This is a complex of graded $GL(W)$-equivariant modules over the graded algebra $\Sym(W)$, with $gr(W) =1$. 
\begin{example}
 In grade zero, the complex is $0 \to \mathbb C$; in grade $1$, it is $W \xrightarrow{\sim} W$.
\end{example}

Consider the $n$-th tensor power of the complex $C_W^{\bullet}$ (as a complex of modules over $\Sym(W)$). 

The resulting complex $\left( C_W^{\bullet} \right)^{\otimes_{\Sym(W)} n} $ is isomorphic to the complex $$\Sym W \otimes \left(W^{\otimes n} \otimes \bdel_n^n \right)^{S_n} \rightarrow \ldots \to \Sym W \otimes W \to \Sym W $$ with degree $n-k$ given by $ \Sym W \otimes \left(W^{\otimes k} \otimes \bdel_k^n \right)^{S_k}$ and differential $\partial=\sum_i (-1)^i {\mu}^{(i)} \otimes \res^i$. This complex is graded, with grade $N$ being exactly the complex $SW_N(K_{n, N}^{\bullet})$.

Of course, the complexes $ C_W^{\bullet}$, $\left( C_W^{\bullet} \right)^{\otimes_{\Sym(W)} n} $ are just specializations of the complexes $C^{\bullet}$, $(C^{\bullet})^{\widetilde{\otimes} 
n}$ under the specialization functor $$\V \to \Rep(GL(W))_{pol}, \;\;\; V \mapsto W.$$

In particular, taking the $N$-th grade of the complex $(C^{\bullet})^{\widetilde{\otimes} n}$ corresponds to considering only the $N$-th part of the ${\bf FI}$-module $\Phi(C^{\bullet})^{\widetilde{\otimes} n}$ in the notation of Section \ref{sec:infinity}. This is precisely the complex $K_{n, N}^{\bullet}$.

\section{The \texorpdfstring{complex $K_{n, t}^{\bullet}$ in $\uRep^{ab}(S_t)$}{Delta complex} in the Deligne category}\label{sec:Delta_complex_Deligne}
\subsection{The \texorpdfstring{complex $K_{n, t}^{\bullet}$}{Delta complex}}\label{ssec:Delta_complex_Deligne}
We now define the counterpart $K_{n, t}^{\bullet}$ in $\uRep^{ab}(S_t)$ of $K_{n, \infty}^{\bullet}$:
$$ K_{n, t}^{\bullet} : 0 \to  \left(\bdel_n^n \boxtimes \Delta_n^{t} \right)^{S_n} \to 
\left(\bdel_{n-1}^n \boxtimes  \Delta_{n-1}^{t}\right)^{S_{n-1}}  \to  \ldots 
\to  \Delta_1^{t}  \to   \Delta_0^{t} \to 0 $$
The $(n-k)$-th component of this complex is the $S_n \times 
S_{t}$-module $\left( \bdel_k^n \boxtimes  \Delta_{k}^{t} 
\right)^{S_{k}}$. The differentials in this complex are
$$\partial: \left( \bdel_k^n \boxtimes   \Delta_{k}^{t} \right)^{S_{k}} 
\longrightarrow \left( \bdel_{k-1}^n \boxtimes   \Delta_{k-1}^{t} 
\right)^{S_{k-1}}$$ given by $\partial = \sum_{i=1}^k (-1)^i \res^i \otimes 
\res^i$.

We will show that its cohomology vanishes in degrees higher than $-n$, and compute explicitly its cohomology in degree $-n$. 

\begin{remark}
 A description of a complex very similar to $K_{n, t}^{\bullet}$ in the setting of $\uRep(S_t)$ was given by Deligne in \cite[(7.12)]{Del07} where its Euler characteristic is computed. However, the reader should keep in mind that the complex constructed in \cite[(7.12)]{Del07} is shifted, and hence multiplied by the sign representation of $S_n$.
\end{remark}

Fix a unital vector space $(W,  w_0 \in W\setminus \{0\})$. Let us consider the contravariant Schur-Weyl functor for $(W, w_0)$ from $\uRep^{ab}(S_t)$:
$$ \mathbf{SW}_{t, W}: \uRep^{ab}(S_t)^{op} \longrightarrow \Mod_{\U(\gl(W))}, \;\;\; M \mapsto \Hom_{\uRep^{ab}(S_t)}(M, W^{\underline{\otimes} t}) $$
as defined in Section \ref{ssec:Deligne_SW} and \cite{En}. In fact, the image of this functor lies in the ``mirabolic'' category $\mathcal{O}^{\mathfrak{p}}$ for the Harish-Chandra pair $(\gl(W), \mathfrak{p} = Aut(W, w_0))$. This functor is left-exact, and by the results of \cite{En}, the derived functors $R^i \mathbf{SW}_{t, W}$ applied to $M \in \uRep^{ab}(S_t)$ give finite-dimensional $\gl(W)$-modules. Taking a sequence $U_l = \mathbb{C}^l$ for $l \in \mathbb{Z}_{\geq 0}$, and $W_l = U_l \oplus \mathbb C w_0$, it was explained in Section \ref{ssec:Deligne_SW} and Proposition \ref{prop:SW_Deligne_infty_equivalence} that the functors $ \mathbf{SW}_{t, W_l}$ induce an equivalence between the category $\uRep^{ab}(S_t)^{op}$ and a certain inverse limit of the categories $\mathcal{O}^{\mathfrak{p}_l}$.

In particular, to show that a certain object $M \in \uRep^{ab}(S_t)$ is zero, it is enough to check that for all $l$ large enough, $ \mathbf{SW}_{t, W_l}(M)$ is finite-dimensional. 

So, we will consider the complex $\mathbf{SW}_{t, W}(K_{n, t}^{\bullet})$, whose $k$-th degree is $\mathbf{SW}_{t, W}(K_{n, t}^{-k})$. To show that the cohomology of $K_{n, t}^{\bullet}$ in degrees higher than $(-n)$ vanishes, we just need to show that $\mathbf{SW}_{t, W}(K_{n, t}^{\bullet})$ has finite-dimensional homology in degrees smaller than $n$ when $\dim W >>0$.

\subsection{Image of \texorpdfstring{$K_{n, t}^{\bullet}$}{complex} under Schur-Weyl duality}\label{ssec:Delta_complex_SW}
Fix a finite dimensional unital vector space $W = U \oplus \mathbb C w_0$ with $\dim W >>n$, and let $\mathfrak{p} = Aut(W, w_0)$. 

Consider the composition of the functor $ \mathbf{SW}_{t, W}$ with the restriction functor $\Mod_{\U(\gl(W))} \to \Mod_{\U(\gl(U))}$ to get a functor
$$ \overline{\mathbf{SW}} = Res_{\gl(U)} \mathbf{SW}_{t, W}: \uRep^{ab}(S_t)^{op} \longrightarrow \Ind \mbox{-} \Rep(GL(U))_{pol}$$
 Clearly, $Res_{\gl(U)}$ is faithful and exact, so it is enough to check that $\overline{\mathbf{SW}}(K_{n, t}^{\bullet})$ has finite-dimensional (if fact, we will prove that it is zero) homology in degrees smaller than $n$ when $\dim U >>0$.

In a first step we want to compute the image $\overline{\mathbf{SW}}(K_{n, t}^{\bullet})$  of the Delta complex under $\overline{\mathbf{SW}}$. We abbreviate $A = \Sym(U)$. This is a commutative algebra with a $GL(U)$-action.

\begin{proposition}
 For any $t \in \mathbb{C}$, we have an isomorphism of $GL(U)$-equivariant free $ A$-modules $$\overline{\mathbf{SW}}(\Delta_{k}^{t}) =A \otimes W^{\otimes k}$$ with $W$ considered as a $GL(U)$-module (isomorphic to $U \oplus \mathbb C$), and $A$ acting freely by multiplication on the left. Moreover, this morphism is $S_k$-equivariant, with $S_k$ acting on the RHS by permuting the factors $W$.

In addition, $\mathbf{SW}_{t, W}(\res^i: \Delta_{k}^{t} \to \Delta_{k-1}^{t}) $ is the map $$ \eps^{(i)}: A \otimes W^{\otimes k-1} \longrightarrow A \otimes W^{\otimes k}, \;\;\; \eps^{(i)} = \id \otimes \id^{\otimes i-1}_W \otimes \eps \otimes \id^{\otimes k-i}_W $$ and $\eps: \mathbb{C} \to W$ is the inclusion $1 \mapsto w_0$.
 
\end{proposition}

\begin{proof}
Consider the $\gl(V)$-module $\mathbf{SW}_{t, W}(\X_{\lambda})$ in the parabolic category $\mathcal{O}^{\mathfrak{p}}$ for the Harish-Chandra pair $(\gl(V), \mathfrak{p})$.

 Recall from \cite{En} that when $\dim W >>\abs{\lambda}$, the module $\mathbf{SW}_{t, W}(\X_{\lambda})$ is either isomorphic to a generalized Verma module or is projective in $\mathcal{O}^{\mathfrak{p}}$.  In both cases, it has a standard filtration, and thus is free over the enveloping algebra of the dual unipotent radical $\mathfrak{u}_{\mathfrak{p}}^-$ of $\mathfrak{p}$. Since $\mathfrak{u}_{\mathfrak{p}}^- \cong U$, and $A \cong \U(\mathfrak{u}_{\mathfrak{p}}^-)$, we conclude that the $GL(U)$-equivariant $A$-module $\overline{\mathbf{SW}}(\X_{\lambda})$ is free over $A$. Thus the modules $\mathbf{SW}_{t, W}(\Delta_t^k)$ are free over $A$ as well, for $\dim W >> n \geq k$. 
%
%
%

Consider the $GL(U)$-module $$\overline{\mathbf{SW}}(\Delta_{k}^{t})= \Hom_{\uRep^{ab}(S_t)}(\Delta_{k}^{t}, V^{\underline{\otimes} t})$$ with $V^{\underline{\otimes} t}$ defined as in \eqref{eq:complex_tens_power}. Then
\begin{align*}
& \overline{\mathbf{SW}}(\Delta_{k}^{t})= \Hom_{\uRep^{ab}(S_t)}(\Delta_{k}^{t}, V^{\underline{\otimes} t}) = \bigoplus_{m \geq 0} \Hom_{\uRep(S_t)}\left(\Delta_{k}^{t}, \left( \Delta^t_m \otimes U^{\otimes m}\right)^{S_m}\right) \cong \\
&\cong \bigoplus_{m \geq 0} \left(\C \bar{P}_{k, m} \otimes U^{\otimes m}\right)^{S_m}.
\end{align*}
Here $\C\bar{P}_{k, m} = \Hom_{\uRep(S_t)}(\Delta_{k}^{t},  \Delta^t_m)$ (see Section \ref{ssec:Delta_obj_Deligne}), and $\bar{P}_{k, m}$ is the set of partial pairings of the sets $\{1, \ldots , k \}$, $\{1, \ldots , m\}$ as described in Notation \ref{notn:Delta_maps_Deligne}. The maps $\res^i: \Delta_{k}^{t} \to \Delta_{k-1}^{t}$ induce maps $(\res^i)^*:\C\bar{P}_{k-1, m} \to \C\bar{P}_{k, m}$ given by composition with $\res^i$.

Note that the set of partial pairings $\bar{P}_{k, m}$ is in fact determined by fixing two injections of a set $\{1,\ldots, l\}$ into the sets $\{1, \ldots , k \}$, $\{1, \ldots , m\}$, for some $l \leq k, m$. This gives us a decomposition of $S_k \times S_m$-modules $$\C\bar{P}_{k, m} \cong \bigoplus_{l \geq 0} \left( \Delta^k_l \otimes \Delta^m_l \right)^{S_l}$$ and $(\res^i)^*$ corresponds to a map $$(\res^i)^*:\bigoplus_{l \geq 0} \left( \Delta^{k-1}_l \otimes \Delta^m_l \right) \to \bigoplus_{l \geq 0} \left( \Delta^k_l \otimes \Delta^m_l \right)$$ induced by the map 
$$(\res^i)^*: \Delta^{k-1}_l \to \Delta^{k}_l, \;\; (\res^i)^*(f) = \sum_{\substack{g \in \Inj([l], [k]):\\ g \circ \iota_i =f}} g$$
for any $f \in \Inj( [l],[k-1])$ and $\iota_i$ defined in (\ref{eq:iota}).

Hence
\begin{align*}
& \overline{\mathbf{SW}}(\Delta_{k}^{t})\cong \bigoplus_{m \geq 0} \left(\bigoplus_{l \geq 0} \left( \Delta^k_l \otimes \Delta^m_l \right)^{S_l} \otimes  U^{\otimes m}\right)^{S_m} \cong  \bigoplus_{l \geq 0} \left( \Delta^k_l \otimes\bigoplus_{m \geq 0} \left(\Delta^m_l  \otimes  U^{\otimes m}\right)^{S_m}\right)^{S_l}
\end{align*} 
with $$ \mathbf{SW}_{t, W}(\res^i) = \oplus_{l \geq 0} (\res^i)^* \otimes \id.$$

Now, by Lemma \ref{lem:SW_funct_class_Delta}, $$\bigoplus_{m \geq 0} \left(\Delta^m_l  \otimes  U^{\otimes m}\right)^{S_m} \cong \Sym(U) \otimes U^{\otimes l}$$ so we have an isomorphism of $GL(U) \times S_k$-modules
\begin{align*}
 \overline{\mathbf{SW}}(\Delta_{k}^{t})\cong  \bigoplus_{l \geq 0} \left( \Delta^k_l \otimes \Sym(U) \otimes U^{\otimes l}\right)^{S_l}  \cong \Sym(U) \otimes  \bigoplus_{l \geq 0} \left( \Delta^k_l \otimes  U^{\otimes l}\right)^{S_l} \cong \Sym(U) \otimes  W^{\otimes k}
\end{align*} 
(the last isomorphism is due to \eqref{eq:classical_tens_power}).

As we have seen, $A \cong \Sym(U)$ acts freely on this module by multiplication on the left. This implies the first part of the Proposition. 

For the second part, notice that under the isomorphism $$\Sym(U) \otimes \bigoplus_{l \geq 0} \left( \Delta^k_l \otimes  U^{\otimes l}\right)^{S_l} \cong \Sym(U) \otimes  W^{\otimes k}$$ the maps $ \mathbf{SW}_{t, W}(\res^i) = \oplus_{l \geq 0} (\res^i)^* \otimes \id$ on LHS correspond exactly to the maps $ \eps^{(i)}$ on RHS (cf. \cite[Appendix A.3]{En}).

\end{proof}
\begin{corollary}
  The above complex $\overline{\mathbf{SW}}(K_{n, t}^{\bullet})$ is isomorphic to $\left( C'^{\bullet} \right)^{\otimes_{A} n}$ where $C'^{\bullet} :A \to A \otimes W$ (sitting in degrees $0, 1$) is the complex of of free $A$-modules with the differential given by $\id_{A} \otimes \eps$.
\end{corollary}
Using the fact that the map $\id_{A} \otimes \eps$ is split injective, we obtain:
\begin{corollary}\label{cor:Delta_cx_SW_t_exact}
 The complex $\overline{\mathbf{SW}}(K_{n, t}^{\bullet})$ has zero homology in all degrees except $n$, and $$H^{n}\overline{\mathbf{SW}}(K_{n, t}^{\bullet}) \cong A \otimes U^{\otimes n}.$$
\end{corollary}


\begin{proposition}\label{prop:Delta_cx_t_hmgy}
 For $k<n$, $H^{-k} K_{n, t}^{\bullet}=0$, and we have an isomorphism of $S_n$-modules in $\uRep^{ab}(S_t)$ $$H^{-n}(K_{n, t}^{\bullet}) \cong \bigoplus_{\mu \vdash n}\mu^{\vee} \otimes \mathbf{M}_{\mu}.$$ 
\end{proposition}
Here $\mathbf{M}_{\mu}$ is the standard module corresponding to $\mu$ in the highest weight category $\uRep^{ab}(S_t)$ (see Section \ref{ssec:Deligne_abelian_struct}).

\begin{remark}
 In particular, this answers the first question posed by Deligne in \cite[(7.13)]{Del07}. 
\end{remark}

\begin{proof}
As it was explained in Section \ref{ssec:Deligne_SW} and Proposition \ref{prop:SW_Deligne_infty_equivalence}, since the cohomology of $\overline{\mathbf{SW}}(K_{n, t}^{\bullet})$ vanishes in degrees less than $n$, we have $H^{-k} K_{n, t}^{\bullet}=0$ for $k<n$. 

For the second part, we compute the $(-n)$-th cohomology of $K_{n, t}^{\bullet}$. We first analyze its decomposition in the Grothendieck group. For any $\tau$ and $0 \leq k \leq n$, we define $${a}_{\mu, \tau}^k =\sharp \{\lambda \vdash k: \tau\subset \lambda\subset\mu^{\vee}, \, \mu^{\vee} / \lambda \in VS, \, \lambda / \tau \in HS \}.$$
By Lemmas \ref{lem:Delta_classical_decomp}, \ref{cor:Delta_t_decomp_standard} we have the following decomposition in the Grothendieck group of $\uRep^{ab}(S_t)$ for any $k \geq 0$ and any $\mu\vdash n$:
\begin{equation}\label{eq:decomp_cx_Deligne}
\Hom_{S_n}(\mu, K_{n, t}^{-k}) = \Hom_{S_n}(\mu, (\bdel^n_k \otimes \Delta_k^t)^{S_k}) \cong \sum_{\tau}  a_{\mu, \tau}^k [\mathbf{M}_{\tau}]
\end{equation}


Let us compare this with the Euler characteristic $\chi$ of the complex $\Hom_{S_n}(\mu,K_{n, t}^{\bullet})$. By \eqref{eq:decomp_cx_Deligne} the following equality holds in the Grothendieck group of $\uRep^{ab}(S_t)$:
\begin{equation}\label{eq:Euler_char_cx_t}
 \chi(\Hom_{S_n}(\mu,K_{n, t}^{\bullet})) = \sum_{\tau} \sum_{k \geq 0} (-1)^{-k} a_{\mu, \tau}^k [\mathbf{M}_{\tau}]
\end{equation}

Recall from the proof of Proposition \ref{prop:Delta_cx_infty_hmgy} that for any $\mu \vdash n$ and any $\tau$, $$\sum_{k \geq 0} (-1)^{-k} {a}^k_{\mu, \tau} = (-1)^n \delta_{\mu^{\vee}, \tau}.$$

Thus in the Grothendieck group of $\uRep^{ab}(S_t)$, we have: $$\chi(\Hom_{S_n}(\mu,K_{n, t}^{\bullet})) = [\mathbf{M}_{\mu^{\vee}}].$$ 

We now check that $$\Hom_{S_n}(\mu,H^{-n} K_{n, t}^{\bullet}) \cong \mathbf{M}_{\mu^{\vee}}.$$ This is clear whenever $ \mathbf{L}_{\mu^{\vee}}$ sits in a semisimple block, since in this case, $ \mathbf{M}_{\mu^{\vee}} =  \mathbf{L}_{\mu^{\vee}}$.

Now assume $ \mathbf{L}_{\mu^{\vee}}$ sits in a non-semisimple block corresponding to partition $\tau \vdash t$, and $\mu^{\vee} = \tau^{(i)}$ for some $i\geq 0$. If $i=0$, then 
$ \mathbf{M}_{\mu^{\vee}} =  \mathbf{L}_{\mu^{\vee}}$ and again we are done.

Assume that $i>0$. Then we have seen that $\Hom_{S_n}(\mu,H^{-n} K_{n, t}^{\bullet})$ is an extension of $\mathbf{L}_{\tau^{(i)}}$, $\mathbf{L}_{\tau^{(i-1)}}$. 
Now, by Lemma \ref{lem:Delta_t_decomp}, $$\Hom_{S_n}(\mu,K_{n, t}^{\bullet}) \cong \sum_{ \tau' \in B^t_{\mu^{\vee}}} [\X_{\tau'}]$$
In particular, by Remark \ref{rmk:Delta_t_decomp}, $\Hom_{S_n}(\mu,K_{n, t}^{\bullet})$ contains exactly one indecomposable summand which lies in the block corresponding to $\tau$, and that is $\X_{\tau^{(i)}} = \mathbf{P}_{\tau^{(i-1)}}$ (the injective hull and projective cover of $\mathbf{L}_{\tau^{(i-1)}}$). 

Hence $\Hom_{S_n}(\mu,H^{-n} K_{n, t}^{\bullet})$ is an extension of $\mathbf{L}_{\tau^{(i)}}$, $\mathbf{L}_{\tau^{(i-1)}}$, and a submodule of $\mathbf{P}_{\tau^{(i-1)}}$. There is only one such submodule, and this is $ \mathbf{M}_{\mu^{\vee}} = \mathbf{M}_{\tau^{(i)}}$, which implies:

$$\Hom_{S_n}(\mu,H^{-n} K_{n, t}^{\bullet}) = \mathbf{M}_{\mu^{\vee}} $$ for any $\mu \vdash n$. This completes the proof of the proposition.

\end{proof}


\section{Proof of \texorpdfstring{Theorem \ref{introthrm:main}}{the main Theorem}}\label{sec:proof_main}
In this section, we prove Theorem \ref{introthrm:main}. Let us state it more explicitly:
\begin{theorem}\label{thrm:main}
\mbox{}
\begin{itemize}
 \item The SM functor $\mathbf{\Gamma}_t: \Rep(S_{\infty}) \longrightarrow \uRep^{ab}(S_t)$ is 
faithful and exact.
 \item On objects, the functor $\mathbf{\Gamma}_t$ acts by $$\mathbf{\Gamma}_t(L_{\lam}) \cong 
\mathbf{M}_{\lam}, \;\;\; \mathbf{\Gamma}_t(Q_{\lam}) \cong \bigoplus_{\mu \in B^t_{\lam}} 
\X_{\mu}  $$ where $\mathbf{M}_{\lam}$ is the standard object corresponding to 
$\lam$ ($\mathbf{M}_{\lam} = \X_{\lam}$ if $t \notin \Z_{\geq 0}$), and 
$$B^t_{\lam} = \{ \mu \subset\lam \rvert \lam / \mu \in HS, \; \text{ and 
either } \; \X_{\mu} \text{ is in a semisimple block, or } \mu = \mu'^{(i)}, \; 
\mu'^{(i+1)} \not\subset \lam\}$$ (notice that for $t \notin \Z_{\geq 0}$, 
this reduces to $B^t_{\lam} = \{ \mu \subset\lam \rvert \lam / \mu \in HS 
\}$).
\end{itemize}
\end{theorem}

\begin{proof}
Fix $\lambda$ to be a partition, and let $n = \abs{\lambda}$.

First, we prove that $\mathbf{\Gamma}_t(Q_{\lam}) \cong \bigoplus_{\mu \in B^t_{\lam}} \X_{\mu}  $. Indeed, by Lemma \ref{lem:Delta_decomposition_infty}, $Q_{\lam} \cong \Hom_{S_n}(\lambda, \Delta_n^{\infty})$. Now, by Lemmas \ref{lem:Delta_objects_corresp} and \ref{lem:Delta_t_decomp} respectively, we have
$$\mathbf{\Gamma}_t(Q_{\lam}) \cong \Hom_{S_n}(\lambda, \Delta_n^{t}) \cong \bigoplus_{\mu \in B^t_{\lam}} \X_{\mu}. $$
Next, we show that the functor $\mathbf{\Gamma}_t$ is exact. In order to do this, we only need to check that for every simple object $L_{\lambda}$ in $\Rep(S_{\infty})$ there is an injective resolution $L_{\lambda} \to I^{\bullet}_{\lambda}$ such that the cohomology of $\mathbf{\Gamma}_t(I^{\bullet}_{\lambda})$ vanishes in degrees higher than $-n$. 

Recall from Corollary \ref{cor:Delta_cx_infty_resol} that $I^{\bullet}_{\lambda}=\Hom_{S_n}( \lambda^{\vee}, K_{n, \infty}^{\bullet})$ is an injective resolution of $L_{\lambda}$. By Lemmas \ref{lem:Delta_objects_corresp}, \ref{lem:res_morph_corresp}, $\mathbf{\Gamma}_t(K_{n, \infty}^{\bullet})$ is isomorphic to the complex $K_{n, t}^{\bullet}$; the cohomology of the latter vanishes in degrees higher than $-n$ by Proposition \ref{prop:Delta_cx_t_hmgy}. Thus $\mathbf{\Gamma}_t$ is exact.

Moreover, Proposition \ref{prop:Delta_cx_t_hmgy} also tells us that $H^{-n} \left( \Hom_{S_n}( \lambda^{\vee}, K_{n, t}^{\bullet}) \right) = \mathbf{M}_{\lam}$, which implies that $\mathbf{\Gamma}_t(L_{\lam}) \cong 
\mathbf{M}_{\lam}$.

In particular, $\mathbf{\Gamma}_t(L_{\lam}) \neq 0$ for any $\lambda$; an exact functor which does not annihilate any simple object is necessarily faithful, which completes the proof of the theorem.

\end{proof}

\begin{remark} The theorem implies that the essential image of the functor $\mathbf{\Gamma}_t$ consists of standardly-filtered
objects. It has been pointed out to us that there is a similarity between this result and Putman's theorem \cite[Theorem E]{Putman} on the Specht stability of centrally stable sequences. It would be interesting to develop an explicit connection between these results by constructing a correspondence between the algebraic representations of $S_{\infty}$ and the Specht-stable sequences of $S_n$-representations over $\mathbb{F}_p$ via a reduction of sorts, and then using Harman's methods \cite{Harman} to see that the latter correspond to standardly filtered objects in $\uRep^{ab}(S_t)$. We do not know, however, a functorial way to do this, and do not expect this method to give an alternative proof of our results.
\end{remark}

\begin{remark} As shown in remark \ref{rem:not-injective}, the functor $\mathbf{\Gamma}_t$ is not injective on objects. Here is a way to produce many examples. By \cite{SS-2} $\dim \Ext_{\Rep(S_{\infty})}^1(L_{\lambda},L_\mu) = 1$ if and only if $\lambda / \mu$ is a vertical strip (and $0$ otherwise). Now take $L_{\lambda}, L_{\mu}$ in $\Rep (S_{\infty})$ such that $\lambda / \mu$ is a vertical strip and denote by $E$ the corresponding non-split extension in $\Rep (S_{\infty})$. Then we choose $t$ in such a way that $\mathbf{\Gamma}_t (L_{\lambda})$ is a semisimple object in $\uRep^{ab}(S_d)$. Then $\mathbf{\Gamma}_t (E) = \mathbf{\Gamma}_t (L_{\lambda} \oplus L_{\mu})$. Alternatively we could take a simple object in $\Rep(S_{\infty})$ with many extensions (e.g. $\lambda = (5,4,3,2,1)$ with non-split extensions with 5 other simple modules). Then some of these extensions have to decompose under $\mathbf{\Gamma}_t$. Note also that $\Rep( S_{\infty})$ has wild representation type whereas $\uRep^{ab}(S_t)$ has tame representation type by \cite{CO, En}.
\end{remark}

\begin{remark} \label{rem:standard} The essential image of $\mathbf{\Gamma}_t$ does not contain all standardly filtered objects. Each block contains standardly filtered indecomposable objects of Loewy length 2 with standard objects $\mathbf{M}_{\tau^{(i)}}, \mathbf{M}_{\tau^{(i+2)}}, \ldots, \mathbf{M}_{\tau^{(i+2r)}}$ for some $i,r$. Consider as an example the indecomposable object $I$ in a block $\{ \tau^{(i)} \}_{i \geq 0}$ of $\uRep^{ab}(S_t)$ for $t \in \Z_{\geq 0}$ with socle layers (socle below, top above)

\[ \begin{pmatrix} \mathbf{L}_{\tau^{(0)}} & & \mathbf{L}_{\tau^{(2)}} \\ & \mathbf{L}_{\tau^{(1)}} & \end{pmatrix}. \]

Then it has a standard filtration with the standard objects $\mathbf{L}_{\tau^{(0)}} = \mathbf{M}_{\tau^{(0)}}$ and $\mathbf{M}_{\tau^{(2)}}$. If it would be in the image of $\mathbf{\Gamma}_t$, it should therefore come from an extension $E$ in $\Rep (S_{\infty})$

\[ \xymatrix{ 0 \ar[r] & L_{\tau^{(0)}} \ar[r] & E \ar[r] & L_{\tau^{(2)}} \ar[r] & 0 } \]

Let us now assume that $t=3$. Then one has 3 non-semisimple blocks, one of them starting with $\tau^{(0)} = \emptyset, \tau^{(1)} = (4), \tau^{(2)} = (4,1)$ \cite[Example 5.10]{CO}. But then the above extension has to split in $\Rep (S_{\infty})$ since $\tau^{(2)} / \tau^{(0)}$ is not a vertical strip and hence there is no non-split extension between $L_{\tau^{(2)}}$ and $L_{\tau^{(0)}}$ in $\Rep (S_{\infty})$. Hence $\mathbf{\Gamma}_3 (E) = \mathbf{M}_{\tau^{(0)}} \oplus \mathbf{M}_{\tau^{(2)}}$ and the indecomposable object $I$ cannot be the image of an indecomposable element under $\mathbf{\Gamma}_3$. 
\end{remark}

\begin{remark} Consider $Q_{\emptyset} \in \Rep (S_{\infty})$ and its image under $\mathbf{\Gamma}_3$. Then $\mathbf{\Gamma}_3 (Q_{\emptyset})$ has $ X_{\emptyset} \cong \triv$ as a direct summand. Since $\triv$ is not injective nor projective in $\uRep^{ab}(S_3)$, we conclude that $\mathbf{\Gamma}_t$ does not necessarily preserve injective objects. 
\end{remark}

\section{Structure constants in the Grothendieck ring of \texorpdfstring{$\uRep^{ab}(S_t)$}{Deligne category}}\label{sec:Grothendieck}

Consider the following corollary of Theorem \ref{thrm:main}:
\begin{corollary}
 The subcategory of standardly-filtered objects in $\uRep^{ab}(S_t)$ is closed under tensor products, and the multiplicity of $\mathbf{M}_{\tau}$ in the standard filtration of 
 $\mathbf{M}_{\lambda} \otimes \mathbf{M}_{\mu}$ is the reduced Kronecker coefficient $\bar{g}^{\tau}_{\lambda, \mu}$. 
\end{corollary}
\begin{proof}
 By Theorem \ref{thrm:main}, this multiplicity equals the multiplicity of $L_{\tau}$ in $L_{\lambda} \otimes L_{\mu}$ (object of $\Rep (S_{\infty})$), which is known to be $\bar{g}^{\tau}_{\lambda, \mu}$.
\end{proof}

We now compute the structure constants in the Grothendieck ring of $\uRep^{ab}(S_t)$.

Fix $t \in \C$. We will use the following notation to shorten our formulas: given a partition $\lambda$ and $ k\in \Z$, we denote by $\mathbf L_{\lambda^{\dagger k}}$ (resp. $\mathbf M_{\lambda^{\dagger k}}$) the following object in $\uRep^{ab}(S_t)$:

\begin{itemize}
 \item If $\mathbf X_\lambda$ belongs in semisimple block of $\uRep^{ab}(S_t)$, then $$\mathbf M_{\lambda^{\dagger k}} = \mathbf L_{\lambda^{\dagger k}} = \mathbf L_{\lambda} = \mathbf X_{\lambda}$$ for $k=0$, and $\mathbf M_{\lambda^{\dagger k}} = \mathbf L_{\lambda^{\dagger k}} =0$ if $k \neq 0$.
 \item Otherwise, let $\lambda = \lambda'^{(i)}$ for some $i \geq 0$ and $\lambda'$ such that $\lambda'[t]$ is a partition (see \eqref{eq:block_Deligne}). Then $$\mathbf M_{\lambda^{\dagger k}} =\mathbf M_{\lambda'^{(i+k)}}, \;\; \mathbf L_{\lambda^{\dagger k}} = \mathbf L_{\lambda'^{(i+k)}}$$ where we put $\mathbf{M}_{\lambda^{(j)}} = \mathbf{L}_{\lambda^{(j)}} = 0$ for $j <0$. 
\end{itemize}

In particular, $\mathbf L_{\lambda^{\dagger 0}} = \mathbf L_{\lambda}$ for any $\lambda$ and $t$.

\begin{proposition}
 Given three partitions $\lambda, \mu, \tau$, consider the composition series of $ \mathbf  L_\lambda \otimes \mathbf  L_\mu$ in $\uRep^{ab}(S_t)$. The multiplicity of $\mathbf L_{\tau}$ in this series is then 
 $$\sum_{j, k \geq 0} (-1)^{j+k} \left( \bar{g}^{\tau}_{\lambda^{\dagger (-j)} , \mu^{\dagger (-k)}} + \bar{g}^{\tau^{\dagger 1}}_{\lambda^{\dagger (-j)} , \mu^{\dagger (-k)}} \right)$$
\end{proposition}
\begin{proof}
Using the above notation, we can rewrite the result of \cite[Section 4.4]{En} as follows: for any $t$ and any partition $\lambda$, we have the following equality in the Grothendieck ring of $\uRep^{ab}(S_t)$:
$$[\mathbf M_{\lambda}] = [\mathbf L_{\lambda^{\dagger 0}}] + [\mathbf L_{\lambda^{\dagger (-1)}}]$$
Hence
$$[\mathbf L_{\lambda}] = \sum_{j \geq 0} (-1)^j [\mathbf M_{\lambda^{(\dagger (-j)}}]$$
and thus 
\begin{align*}
&[\mathbf L_{\lambda} \otimes \mathbf L_{\mu}] = [\mathbf L_{\lambda}][\mathbf L_{\mu}]  = \left(\sum_{j \geq 0} (-1)^j [\mathbf M_{\lambda^{\dagger (-j)}}] \right) \left(\sum_{k \geq 0} (-1)^k [\mathbf M_{\mu^{\dagger (-k)}}] \right) = \\
& = \sum_{j, k \geq 0} (-1)^{j+k} [\mathbf M_{\lambda^{\dagger (-j)}}] [\mathbf M_{\mu^{\dagger (-k)}}] = \sum_{j, k \geq 0} (-1)^{j+k} \sum_{\tau} \bar{g}^{\tau}_{\lambda^{\dagger (-j)}, \mu^{\dagger (-k)}}[\mathbf M_{\tau}] = \\
& = \sum_{j, k \geq 0} (-1)^{j+k} [\mathbf M_{\lambda^{\dagger (-j)}}] [\mathbf M_{\mu^{\dagger (-k)}}] = \sum_{\tau} \sum_{j, k \geq 0} (-1)^{j+k} \bar{g}^{\tau}_{\lambda^{\dagger (-j)}, \mu^{\dagger (-k)}}\left( [\mathbf L_{\tau}] + [\mathbf L_{\tau^{\dagger 1}}] \right)
\end{align*}

Hence the multiplicity of $[\mathbf L_{\tau}]$ in $[\mathbf L_{\lambda} \otimes \mathbf L_{\mu}]$ is 
$$\sum_{j, k \geq 0} (-1)^{j+k} \bar{g}^{\tau}_{\lambda^{\dagger (-j)}, \mu^{\dagger (-k)}} + \bar{g}^{\tau^{\dagger 1}}_{\lambda^{\dagger (-j)}, \mu^{\dagger (-k)}}$$
as required.
\end{proof}


\begin{thebibliography}{999999}

\bibitem[CEF]{CEF} T. Church, J. Ellenberg, B. Farb, {\it FI-modules and stability for representations of symmetric groups.}, Duke Math. J. 164, (2015).
\bibitem[CO]{CO} J. Comes, V. Ostrik, {\it On blocks of Deligne's category 
$\uRep(S_t)$}, Advances in Mathematics 226 (2011), no. 2, 1331-1377; 
arXiv:0910.5695 [math.RT].
\bibitem[CO1]{CO1} J. Comes, V. Ostrik, {\it On Deligne's category 
$\uRep^{ab}(S_d)$}, Algebra and Number Theory, Vol. 8 (2014), No. 2, 
pp. 473--496; arXiv:1304.3491v1 [math.RT].
\bibitem[CPS]{CPS} E. Cline, B. Parshall, L. Scott, {\it Finite dimensional algebras and highest weight categories}, J. Reine Angew. Math. 391, (1988)
\bibitem[D]{Del07} P. Deligne, {\it La categorie des representations du groupe 
symetrique $S_t$, lorsque $t$ n'est pas un entier naturel}, Algebraic groups 
and homogeneous spaces, Tata Inst. Fund. Res. Stud. Math., p. 209-273, Mumbai 
(2007); http://www.math.ias.edu/~phares/deligne/preprints.html.
\bibitem[En]{En} I. Entova-Aizenbud, {\it Schur-Weyl duality for Deligne 
categories}, IMRN, Vol. 2015, No. 18, pp. 8959–-9060; arXiv:1403.5509 [math.RT].
\bibitem[En2]{En2} I. Entova-Aizenbud, {\it Deligne categories and reduced Kronecker coefficients}, Journal of Algebraic Combinatorics, Volume 44 (2016), Issue 2, pp. 345362; arXiv:1407.1506 (numbering in references refers to the Arxiv version).
\bibitem[En3]{En3} I. Entova-Aizenbud, {\it Schur-Weyl duality for Deligne categories II: the limit case}, Pacific Journal of Mathematics, Vol. 285 (2016), No. 1, pp. 185–224; arXiv:1504.01519 [math.RT], (2015).
\bibitem[EHS]{EHS} I. Entova-Aizenbud, V. Serganova, V. Hinich, {\it Deligne 
categories 
and the limit of categories $Rep(GL(m|n))$}, arXiv:1511.07699 [math.RT], (2015). to appear in IMRN.
\bibitem[EGNO]{EGNO} P. Etingof, S. Gelaki, D. Niksych, V. Ostrik, {\it Tensor 
Categories}, Mathematical Surveys and Monographs, Vol. 205, AMS (2015).
\bibitem[Ha16]{Harman}
{N. Harman},
{\it Deligne categories as limits in rank and characteristic},
{arXiv:1601.03426},
(2016).
\bibitem[M1]{Mu1} F. D. Murnaghan, {\it The Analysis of the Kronecker Product of Irreducible Representations of the
Symmetric Group}, Amer. J. Math., 60(3):761-784, (1938).
\bibitem[M2]{Mu2} F. D. Murnaghan, {\it On the analysis of the Kronecker product of irreducible representations of
$S_n$}, Proc. Nat. Acad. Sci. U.S.A., 41:515-518, (1955).
\bibitem[Pu15]{Putman}
{Andrew {Putman}},
{\it Stability in the homology of congruence subgroups.},
{Invent. Math.},
{202},
{3},
{(2015)}.
\bibitem[SS]{SS} S. Sam, A. Snowden, {\it Stability patterns in representation 
theory}, Forum Math. Sigma, Vol. 3 (2015), , arXiv:1209.5122 [math.AC].
\bibitem[SS2]{SS-2} S. Sam, A. Snowden, {\it GL-equivariant modules over 
polynomial rings in infinitely many variables.}, Trans. Am. Math. Soc., 368 
(2016).
\bibitem[SS3]{SS-3} S. Sam, A. Snowden, {\it Introduction to twisted commutative algebras}, arXiv:1209.5122 [math.RT],
(2012).  	
\bibitem[Ser]{Ser} A. N. Sergeev, {\it The tensor algebra of the identity representation as a module over the lie superalgebras $ \mathfrak{gl}(n,\,m)$ and $ q(n)$},
Mathematics of the USSR-Sbornik, Volume 51, Number 2, (1985).

\end{thebibliography}
\end{document}